\documentclass[a4paper,11pt,twoside,reqno]{amsart}

\usepackage[utf8]{inputenc}
\usepackage[plainpages=false,pdfpagelabels=true]{hyperref}
\usepackage{amssymb,amsthm}
\usepackage[margin=1in]{geometry}
\usepackage{slashed}
\usepackage{caption}
\usepackage[all]{xy}
\usepackage{graphicx}
\usepackage{comment}

\newtheorem{Satz}{Theorem}[section]
\newtheorem{Prop}[Satz]{Proposition}
\newtheorem{Lem}[Satz]{Lemma}

\newtheorem{Thm}[Satz]{Theorem}
\newtheorem{Cor}[Satz]{Corollary}
\theoremstyle{definition}
\newtheorem{Dfn}[Satz]{Definition}
\newtheorem{Bem}[Satz]{Remark}

\newtheorem{Ass}[Satz]{Assumption}
\newtheorem{question}[Satz]{Question}
\newcommand{\tr}{\operatorname{Tr}}

\newcommand{\C}{{\mathbb{C}}}

\newcommand{\T}{\mathrm{T}}

\newcommand{\ag}{\mathfrak{g}}
\newcommand{\ah}{\mathfrak{h}}
\newcommand{\ak}{\mathfrak{k}}

\newcommand{\abs}[1]{\vert #1\vert}

\newcommand{\nor}{\mathrm{nor}}

\newcommand{\eins}{{\mathchoice
{\mathrm 1\mskip-4.2mu\mathrm l}{\mathrm 1\mskip-4.2mu\mathrm l}
{\mathrm 1\mskip-3.9mu\mathrm l}{\mathrm 1\mskip-4.0mu\mathrm l}}}
\newcommand{\Id}{\operatorname{Id}}

\newcommand{\norm}[1]{\left \lVert #1 \right \rVert}
\allowdisplaybreaks[1]

\renewcommand{\epsilon}{\varepsilon}

\newcommand{\R}{\ensuremath{\mathbb{R}}}

\newcommand{\s}{\ensuremath{\mathbb{S}}}
\numberwithin{equation}{section}

\usepackage{color}

\newcommand{\Adj}{\ensuremath{\operatorname{Ad}}}
\newtheorem{innercustomgeneric}{\customgenericname}
\providecommand{\customgenericname}{}
\newcommand{\newcustomtheorem}[2]{%
  \newenvironment{#1}[1]
  {%
   \renewcommand\customgenericname{#2}%
   \renewcommand\theinnercustomgeneric{##1}%
   \innercustomgeneric
  }
  {\endinnercustomgeneric}
}

\newcustomtheorem{customthm}{Theorem}
\newcustomtheorem{customcor}{Corollary}

\title{The Initial Value Problem
for Harmonic maps of Cohomogeneity One manifolds}

\author{Anna Siffert}
\address{Universität M\"unster, Mathematisches Institut\\
Einsteinstr. 62\\
48149 M\" unster\\
Germany}
\email{asiffert@uni-muenster.de}

\thanks{Anna Siffert thanks the German Research Foundation (DFG) for funding - Project-ID 427320536 - SFB 1442.
This project is part of Anna Siffert's project on the construction of harmonic maps between cohomogeneity one manifolds
from the SFB-funding period
2020-2024.
}

\begin{document}
\begin{abstract}
We set up and solve the initial value problem
for equivariant harmonic maps of cohomogeneity one manifolds, i.e. we show the local existence of a harmonic map in the neighborhood of a singular orbit. 
\end{abstract} 
\maketitle

\section{Introduction}
Harmonic maps between Riemannian manifolds and their generalizations are an active research area within geometric analysis, see e.g. the recent manuscripts \cite{MR4927650, BS1} and the references therein, as well as the book \cite{MR2044031}.

\smallskip

A smooth map $\varphi:(M,g)\rightarrow (N,h)$ between Riemannian manifolds is called harmonic if it is a critical point of the energy functional. 
In the following section, we provide a more detailed exposition of harmonic maps, closely following \cite{MR2389639}:
Let $\Omega\subset M$ be a domain with piecewise $C^1$-boundary $\partial\Omega$.
The energy integral of $\varphi$ over the domain $\Omega$ is defined as
\begin{align}
\label{energy}
    E_{\Omega}(\varphi)=\frac{1}{2}\int_{\Omega}\lvert d\varphi\rvert^2 v_g,
\end{align}
where $v_g$ denotes the volume form of $M$ with respect to the metric $g$.
Let $\Phi:M\times (-\epsilon,\epsilon)\rightarrow N, (x,t)\mapsto \Phi(x,t)=:\varphi_t(x)$, $\epsilon\in\mathbb{R}$ with $\epsilon>0$, be a smooth map with $\Phi(x,0)=\varphi(x)$, i.e. $\Phi$ is a one-parameter variation of $\varphi$. The map $\Phi$ is called supported in $\Omega$, if $\varphi_t=\varphi$ on $M\setminus\Omega^{\circ}$, where $\Omega^{\circ}$ denotes the interior of $\Omega$.
A smooth map $\varphi:(M,g)\rightarrow (N,h)$ is called harmonic map if it is critical point of (\ref{energy}). This means that for all compact domains $\Omega$ and all smooth one-parameter variations $\Phi$ of $\varphi$ supported in $\Omega$, we have
\begin{align*}
    \frac{d}{dt}E_{\Omega}(\varphi_t)_{\lvert t=0}=0,
\end{align*}
i.e. the first variation of the energy functional vanishes.

\smallskip

Straightforward computations yield that the first variation of the energy functional is given by
\begin{align*}
    \frac{d}{dt}E_{\Omega}(\varphi_t)_{\lvert t=0}=-\int_{M}\langle \tau(\varphi),\frac{\partial\varphi_t}{\partial t}_{\lvert t=0}\rangle v_g,
\end{align*}
where $\langle\,\cdot\,,\,\cdot\,\rangle$ denotes the inner product on $\varphi^{-1}TN$ which is induced from $h$ and
the tension field of $\varphi$, denoted by $\tau(\varphi)$, is defined by
\begin{align*}
 \tau(\varphi)=\sum_{i=1}^{\dim M}({}^{\varphi}\nabla_{e_i}(d\varphi\, e_i)-d\varphi(\nabla_{e_i}e_i)).   
\end{align*}
Here $\nabla$ denotes the Levi-Civita connection of $(M,g)$,
${}^{\varphi}\nabla$ is the pull-back of the Levi-Civita connection of $(N,h)$ to the bundle $\varphi^{-1}TN$ and $(e_i)_{i=1}^{\dim M}$ is a local orthonormal frame on $M$.
Clearly, $\varphi$ is harmonic if and only if 
\begin{align}
\label{tension}
    \tau(\varphi)=0.
\end{align}
In local coordinates the identity (\ref{tension}) constitutes 
a semilinear elliptic system of partial differential equations of second order.
There is no general solution theory for such systems, therefore solving the equation (\ref{tension}) is often a challenging undertaking.
A popular technique for developing solutions is to impose symmetry criteria so that the original problem is reduced to a simpler, and in the best-case scenario, a (partially) solvable, problem. 
Of course, one challenge is determining the appropriate symmetry conditions to construct a (partially) solvable problem.

\smallskip

One fruitful symmetry assumption in the context of harmonic maps is 
to study
equivariant maps between cohomogeneity one manifolds
for harmonicity.
Urakawa \cite{MR1214054} pioneered the research of equivariant harmonic maps between cohomogeneity one manifolds, which has been widely extended by P\"uttmann and Siffert in \cite{MR4000241}.
Using this Ansatz, examples of harmonic maps between cohomogeneity one manifolds have been produced for specific choices of these cohomogeneity one  manifolds, see e.g. \cite{MR1436833, MR4000241, MR3427685, MR3745872, MR1214054}.

Recall that a manifold $M$ is termed a cohomogeneity one manifold if it admits a smooth isometric action of a compact Lie group $G$, represented as $G\times M\rightarrow M$, for which the orbit space $M/G$ is one-dimensional. Orbits can be classified into two categories: Regular and non-regular orbits.
Regular orbits, known as principal orbits, have a codimension of one. Their union forms an open dense subset $M_0$ within $M$. A singular orbit is defined as a non-regular orbit with a codimension that exceeds one.
Boundary points of $M/G$ are associated with non-regular orbits; specifically, the non-regular orbits represent the fibres of these boundary points under the natural projection $\pi: M\rightarrow M/G$.
Also, remember that equivariant maps are those that preserve the structure of cohomogeneity one manifolds, meaning they map orbits to orbits.

For equivariant maps $\varphi$ between cohomogeneity one manifolds
equation (\ref{tension}) simplifies to a singular boundary value problem. More precisely one obtains an ordinary differential equation of second order for a smooth function $r:(M/G)^{\circ}\rightarrow\mathbb{R}$ of the form 
\begin{align}
\label{ode0}
 \ddot r(t)+h_1(t)\dot r(t)+h_2(t,r(t))=0,   
\end{align}
where $h_1:(M/G)^{\circ}\rightarrow\mathbb{R},\, h_2:(M/G)^{\circ}\times\mathbb{R}\rightarrow\mathbb{R}$ are smooth maps and  
$(M/G)^{\circ}$ denotes the interior of $M/G$.
For $\varphi$ to be harmonic, $r$ must meet certain boundary conditions at the boundary points of $M/G$.

At the boundary points of $M/G$ that correspond to singular orbits, the differential equation (\ref{ode0}) exhibits singular behavior. The differential equation (\ref{ode0}) is regular at points of $M/G$ where the fibers under the natural projection $\pi:M\rightarrow M/G$ are not singular orbits.
In summary, constructing equivariant harmonic maps between cohomogeneity one manifolds requires solutions to (\ref{ode0}) that extend smoothly to the boundary points of $M/G$.
We will address the corresponding initial value problem in this manuscript, assuming that $0\in M/G$ is a boundary point. We will demonstrate that for every $v$, there is a unique smooth solution $r:(M/G)^{\circ}\cup\{0\}\rightarrow\mathbb{R}$ of (\ref{ode0}) such that $r(0)=0$ and $\dot r(0)=v$, which has a continuous dependence on $v$. 
As this is a local problem, it suffices to investigate it in the vicinity of just one singular orbit.

\begin{Thm}
Let a cohomogeneity one manifold $M$ with a singular orbit $N$ be given. Assume that $N$ is the 
fibre of $0\in M/G$ under the natural projection map $\pi:M\rightarrow M/G$.
For each $v\in\mathbb{R}$ there
exists a unique smooth solution $r:(M/G)^{\circ}\cup\{0\}\rightarrow\mathbb{R}$ of (\ref{ode0}) with 
$r(0)=0$ and $\dot r(0)=v$
which depends continuously on $v$. 
\end{Thm}

This result is formulated precisely in Section\,\ref{sec-ivp}. The initial value problem for harmonic maps of cohomogeneity one manifolds has been solved in the literature for specific cases, as seen in \cite{MR1298998, MR2022387,MR3427685, MR3745872}.

\smallskip

Regular-singular systems of first order are examined in Section\,\ref{sec4}.
Such systems arise naturally when dealing with the initial value problem for harmonic maps between cohomogeneity one manifolds.
Many of the results shown in Section\,\ref{sec4} are already known, but they can be found scattered throughout the literature. We present an overview that we hope will be beneficial for future research.

\medskip

\textbf{Organization}: 
In Section\,\ref{sec:prelim} we provide preliminaries.
We provide the tension field for equavariant smooth maps between cohomogeneity one manifolds in Section\,\ref{sec-har}.
The initial value problem
for harmonic maps of cohomogeneity one manifolds is set up and solved in Section\,\ref{sec-ivp}.
Finally, in Section\,\ref{sec4}, we study
regular-singular systems of first order.

\section{Preliminaries}
\label{sec:prelim}
We provide preliminaries on cohomogeneity one manifolds, on smoothness conditions for cohomogeneity one metrics and on harmonic self-maps between Riemannian manifolds.

\subsection{Cohomogeneity one manifolds}
In this subsection we briefly recall cohomogeneity one manifolds.
This subsection adheres closely to the presentation in \cite{MR4000241}, with a few sentences copied nearly verbatim.

\medskip
Let $M$ be a connected manifold throughout.
$M$ is said to have cohomogeneity one if it admits a smooth isometric action of a compact Lie group $G$, represented as $G\times M\rightarrow M$, for which the orbit space $M/G$ is one-dimensional. 
In this case, $M$ is also called a cohomogeneity one manifold.

Mostert \cite{MR95897,MR85460} proved a classification result of cohomogeneity one manifolds. Specifically, he established that the orbit space $M/G$ of these manifolds is isometric to precisely one of the following spaces:
\begin{enumerate}
    \item the real line, i.e. $M/G=\mathbb{R}$;
    \item the circle, i.e. $M/G=S^1$;
    \item a closed interval, i.e. $M/G=[0,L]$ for some $L\in\R_+$; 
    \item a half line, i.e. $M/G=[0,\infty)$.
\end{enumerate}
We denote by $\partial(M/G)$ the boundary points of the orbit space $M/G$, i.e. $\partial(\mathbb{R})=\emptyset$, $\partial(S^1)=\emptyset$, $\partial([0,L])=\{0,L\}$ and $\partial([0,\infty))=\{0\}$.

Let $M$ be a cohomogeneity one manifold.
Every boundary point, meaning every element of $\partial(M/G)$, is associated with a non-principal orbit $N$.
In scenario (3), there are precisely two non-principal orbits, which we refer to as $N_0$ and $N_1$.
In scenario (4), there exists precisely one non-principal orbit, which we refer to as $N_0$.
 Moreover, every interior point of $M/G$ is associated with a principal orbit.
An orbit $N$ that is not principal is termed an exceptional orbit, or just exceptional, if its dimension matches the common dimension of the principal orbits.
In other cases, we refer to $N$ as a singular orbit, or just singular.
The set $M_0$, which consists of all principal orbits, is referred to as the regular part of $M$. 

\smallskip

Let $\gamma$ be a unit-speed normal geodesic, this means a geodesic $\gamma:\R\to M$ that passes through all orbits perpendicularly and satisfies $\gamma(0)\in N_0$ and $\gamma(L) \in N_1$ in case (3), and $\gamma(0)\in N_0$ in case (4).
The isotropy groups of the regular points $\gamma(t)$, $t\in (M/G)^{\circ}$, are constant. We will denote this common principal isotropy group by $H$ from now on. 
In our context, normal geodesics of cohomogeneity one manifolds are significant because they allow for the reduction of geometric problems on these manifolds to an examination of these issues along a specified normal geodesic.

In case (3), we make the additional assumption that $\gamma$ is closed, which is equivalent to the Weyl group $W$ being finite.
Recall that the Weyl group $W$ is defined as the subgroup of elements of $G$ that leave $\gamma$ invariant modulo the subgroup of elements that fix $\gamma$ pointwise. It is a dihedral subgroup of $N(H)/H$, generated by two involutions that fix $\gamma(0)$ and $\gamma(L)$, respectively, and it acts simply transitively on the regular segments of the normal geodesic.

\smallskip

In this manuscript we investigate the local structure of a cohomogeneity one manifold in the vicinity of a singular orbit $N$. Let $\gamma$ be a normal geodesic with unit speed, and let $x_0:=\gamma(0)\in N$. Let $K$ be the isotropy group at $x_0$. Therefore, $N$ is equal to $G\cdot x_0$ and can be identified with $G/K$.

\smallskip

The Riemannian metric \( g \) on \( M \) is characterized by a one-parameter family of \( G \)-invariant metrics \( g_t \) on the principal orbits \( G/H \). Specifically, we have \begin{align} \label{metric} g=g_t+dt^2, \end{align} where \( g_t \) must meet invariance conditions in case (2) and smoothness conditions in cases (3) and (4), as detailed in \cite{MR1923478} and its references. The next subsection addresses the smoothness conditions for cases (3) and (4).
Keep in mind that, due to the density of $M_0$ in $M$, (\ref{metric}) also defines the metric on all of $M$.

\subsection{Smoothness conditions of cohomogeneity one metrics}
Let $M$ be a manifold of cohomogeneity one that has (at least one) singular orbit. Let us further assume that the set $M_0\subset M$, which is both dense and open in $M$, possesses a $G$-invariant metric.
The studies in \cite{MR1758585} and \cite{MR4400726} examined the conditions for the extension of this metric to the singular orbit to be smooth, and they provided smoothness criteria for cohomogeneity one metrics. 
 One can often apply the smoothness conditions described by Verdiani and Ziller in \cite{MR4400726} more easily in practice than those in \cite{VZ1}.  The subsequent presentation adheres closely to \cite{MR4400726}.

\smallskip

A non-compact cohomogeneity one manifold with a singular orbit is a vector bundle, whereas a compact cohomogeneity one manifold is the union of two vector bundles, as discussed in \cite{MR1155662,MR3362465} and their references. 
This manuscript will focus on the local structure of a cohomogeneity one manifold close to a singular orbit.
Thus, we can limit ourselves to a single vector bundle.

\smallskip

Let $G$ be a compact Lie group, and let $H$ and $K$, with $G\subset G$, be closed Lie subgroups such that $H\subset K\subset G$ and $K/H=\s^{\ell}$ for some $\ell\in\mathbb{N}=\{1,2,3,\dots\}$.
To put it differently, $K$ acts transitively on $\s^{\ell}$, meaning there is a representation $\rho:K\rightarrow O(\ell+1)$ for which $\rho(K) x=\s^{\ell}$ for every $x\in\s^{\ell}$.
The isotropy group $K$ acts on the sphere $\s^{\ell}$, and this action can be extended to a linear action on the disk $D:=D^{\ell+1}\subset\mathbb{R}^{\ell+1}$.

Consequently, $M=G\times_KD$ represents a homogeneous disc bundle, where the base space is $G/K$ and the structure group is $H$. In this case, the isotropy group $K$ acts on $G\times D$ in the following manner: 
\begin{align*} 
(k, (g,d))\mapsto (gk^{-1}, \rho(k)d).   
\end{align*}
Note further that $G$ acts with cohomogeneity one on $M$ by left multiplication, i.e.
$$(g,[\hat{g},d])\mapsto[g\hat{g},d]$$ for $d\in D$ and $g, \hat{g}\in G$.
The singular orbit $N:=G\times_K\{0\}$ is diffeomorphic to $G/K$.
For sufficiently small $t$, the principal orbits are represented by $P_t:=G\times_K\s^{\ell}_t$. Here, $\s^{\ell}_t\subset D$ denotes a sphere with radius $t$. Every $P_t$ is a diffeomorphic copy of $G/H$.
By means of the exponential map, one can identify the disk $D$ with a manifold that is orthogonal to the singular orbit at the point $\gamma(0)$ in a G-equivariant way. 
Thus $D$ can be considered as slice of the $G$ action on $M$.

\smallskip

To state the main result of \cite{MR4400726}, let $\gamma\colon[0,\infty)\to \mathbb{R}^{\ell+1}$ be a normal geodesic. 
The metric on $M_0$ is determined by its values along $\gamma$.
This is a consequence of the fact that the values and the action of $G$ determine the metric on all of $M_0$.
Let $\mathfrak{g}$ and $\mathfrak{h}$ denote the Lie algebras of $G$ and $H$, respectively. 
Let $Q$ denote a fixed biinvariant metric on the group $G$. Moreover, let $\mathfrak{n}$ denote the orthonormal complement of the Lie algebra $\ah$ within $\ag$.
The complement $\mathfrak{n}$ can be recognized as the tangent space to the regular orbits along $\gamma$ via action fields, specifically $X\in\mathfrak{n} \mapsto X^{\ast}(\gamma(t))$.

\smallskip

Let $X_i$ denote a basis of $\mathfrak{n}$. Then for each $t>0$, $X^*_i(\gamma(t))$ constitutes a basis of $\dot \gamma^\perp(t)\subset T_{\gamma(t)}M$.
The collection of the $r$ functions $g_{ij}(t)=g(X_i^*,X_j^*)_{c(t)},\ i\le j$ also determines the metric.

With this preparation at hand, we can refer to the following theorem that offers smoothness conditions for $g$:

\begin{Thm}[\cite{MR4400726}]
\label{gsmooth}
Let $g_{ij}(t),\ t>0$ be a smooth family of positive definite matrices describing the cohomogeneity one metric on the regular part along a normal geodesic $\gamma(t)$.     Then there exist integers $a_{ij}^k$ and $ d_k$, with $\ d_k\ge 0$,  such that the metric has a smooth extension to all of $M$ if and only if
$$\sum_{i,j} a_{ij}^k\,g_{ij}(t)=t^{d_k}\phi_k(t^2)\quad \text{ for } k=1,\cdots, r, \ \text{ and } t>0 $$
 where $\phi_1,\cdots,\phi_{r}$ are smooth functions defined for $t\ge 0$.
\end{Thm}

Remember that $Q$ represents a predetermined biinvariant metric on $G$.
Regarding $Q$, we examine the subsequent splitting of the Lie algebra $\mathfrak{g}$: \begin{align*} \mathfrak{g}=\mathfrak{h}\oplus\mathfrak{p}\oplus\mathfrak{m}. \end{align*}
Here, $\mathfrak{h}$ denotes the Lie algebra of $H$, while $\mathfrak{h}\oplus\mathfrak{p}$ represents the Lie algebra of $K$.

Theorem\,\ref{gsmooth} in particular yields
\begin{align*}
&g_{\lvert\mathfrak{m}}=A_0+tA_1+t^2A(t),\\
&g_{\lvert\mathfrak{p}}=t^2\Id_{\mathfrak{p}}+t^4B(t),\\
&g_{\lvert\mathfrak{pm}}=t^2C_0+t^3C(t),\\
\end{align*}
where $A_0$, $A_1$ and $C_0$ are constant matrices and $A, B, C$ are smooth matrix valued functions of $t$, see \cite{VZ1}.

\subsection{Harmonic self-maps between cohomogeneity one manifolds}
\label{sub-har}
We present a brief exposition on equivariant harmonic self-maps between cohomogeneity one manifolds. We again borrow our notation and presentation style from \cite{MR4000241}. 

\medskip

Equivariant maps are the structure preserving maps between cohomogeneity one manifolds, i.e. these maps map orbits to orbits.
 In case (3), P\"uttmann \cite{MR2480860} constructed an infinite family of equivariant smooth self-maps of $M$ given by
\begin{align*}
    \psi: M\rightarrow M,\,  g\cdot\gamma(t)\mapsto g\cdot\gamma(kt).
\end{align*}
 Here $\gamma$ is a unit speed normal geodesic such that $\gamma(0)\in N_0$. Furthermore, $k\in\mathbb{Z}$ is of the form $j\abs{W}/2+1$ where $j\in 2\mathbb{Z}$.
Odd integers $j$ may also be permitted under certain conditions on the action $G\times M\rightarrow M$, as detailed in Lemma\,2.1 of \cite{MR2480860}. In accordance with \cite{MR4000241}, the mapping $g\cdot\gamma(t) \mapsto g\cdot\gamma(kt)$ is referred to as the {\em $k$-map} of $M$.
 When the codimensions of the non-principal orbits $N_0$ and $N_1$ are odd, a $k$-map has a degree that is equal to $k$.
  If the codimensions of the non-principal orbits $N_0$ and $N_1$ are not both odd, then the degree of a $k$-map is either 0 or $\pm 1$.

\smallskip

Let $M$ be an arbitrary cohomogeneity one manifold, without the assumption that we are necessarily dealing with case (3), and let $\gamma$ be a fixed but arbitrary normal geodesic. 

Below we study the equivariant self-maps of $M$ given by
\begin{align}
\label{map_psi}
\psi:M\rightarrow M,\, g\cdot\gamma(t)\mapsto g\cdot\gamma(r(t)).
\end{align}
Here, $r:(M/G)^{\circ}\rightarrow \mathbb{R}$ is a smooth function that also meets the boundary conditions in cases (3) and (4). This means that $r$ meets the following requirements: 
\begin{enumerate}
    \item[(i)] the boundary conditions 
$r(0)=0$ and $r(L)=kL$, in case $(3)$. Here $k=j\abs{W}/2+1$ where $j\in 2\mathbb{Z}$ (and possibly $j\in\mathbb{Z}$, see above);
\item[(ii)] the boundary condition 
$r(0)=0$ in case $(4)$.
\end{enumerate}
Analogous to Lemma\,2.1 in \cite{MR2480860}, one proves that in case (4) the maps (\ref{map_psi}) are smooth.
Consequently, the maps (\ref{map_psi}) are smooth in every instance (1)-(4).

\smallskip

P\"uttmann and Siffert \cite{MR4000241} examined the maps (\ref{map_psi}) for harmonicity in case (3). To this end, the tension field of $\psi$ was computed in case $(3)$, as detailed in Theorems A and B of \cite{MR4000241}. 
The proofs of Theorems A and B involve local calculations.
Therefore, we can perform the same calculations for cases (1), (2), and (3).

In order to state these results, we introduce the following notation, which was also used in \cite{MR4000241}. By
\begin{gather*}
  \Pi_t^{r(t)}: T_{\gamma(t)} (G\cdot\gamma(t)) \to T_{\gamma(r(t))} (G\cdot \gamma(r(t))
\end{gather*}
we denote the parallel transport along the normal geodesic~$\gamma$. 
The action field homomorphism, i.e. the map given by $X^{\ast}_{\vert\gamma(t)} \mapsto X^{\ast}_{\vert \gamma(r(t))}$,
constitutes a homomorphism between the two tangent spaces $T_{\gamma(t)} (G\cdot\gamma(t))$ and $T_{\gamma(r(t))} (G\cdot\gamma(r(t)))$. 
We denote by $J_t^{r(t)}$ the endomorphism of $T_{\gamma(t)} (G\cdot\gamma(t))$ obtained by composing the action field homomorphism with $(\Pi_t^{r(t)})^{-1} = \Pi_{r(t)}^t$.

\smallskip

We get the following result for the normal component of the tension field of the maps $\psi$ given in (\ref{map_psi}) (compare Theorem\,A in \cite{MR4000241}):

\begin{Thm}
\label{A}
The normal component of the tension field of $\psi$, see (\ref{map_psi}), is given by
\begin{gather*}
  \tau^{\nor}_{\vert\gamma(t)} = \ddot r(t) - \dot r(t) \tr S_{\vert \gamma(t)}
    + \tr\, (J_t^{r(t)})^{\ast} (\Pi_t^{r(t)})^{-1} S_{\vert \gamma(r(t))} \Pi_t^{r(t)} J_t^{r(t)}
\end{gather*}
for $t\in(M/G)^{\circ}$. Here, $S_{\vert \gamma(t)}$ denotes the shape operator of the orbit $G\cdot \gamma(t)$ at $\gamma(t)$ and $(J_t^{r(t)})^{\ast}$ denotes the adjoint endomorphism of $J_t^{r(t)}$.
\end{Thm}

\smallskip

Keep in mind that $Q$ represents a predetermined biinvariant metric on $G$. Let $\mathfrak{n}$ represent the orthonormal complement of the Lie algebra $\ah$ of the principal isotropy group $H$ within~$\ag$.
We define the metric endomorphisms $P_t : \mathfrak{n} \to \mathfrak{n}$ by \begin{gather*} Q(X, P_t\cdot Y) = g(X^{\ast},Y^{\ast})_{\gamma(t)}. \end{gather*}
The normal component of the tension field can be expressed as follows using this notation:

\begin{Thm}[Theorem 3.4 in \cite{MR4000241}]
\label{tau-nor}
The normal component of the tension field is given by
\begin{gather*}
  \tau^{\nor}_{\vert\gamma(t)} = \ddot r(t) + \tfrac{1}{2}\dot r(t) \tr P_t^{-1}\dot P_t
    - \tfrac{1}{2} \tr P_t^{-1} (\dot P)_{r(t)}.
\end{gather*}
\end{Thm}

We obtain the following result for the tangential component of the tension field of $\psi$ (compare Theorem\,B in \cite{MR4000241}):

\begin{Thm}
\label{thm-b}
The tangential component of the tension field of a $\psi$ is given by
\begin{gather*}
  \tau^{\tan}_{\vert\gamma(t)} = -\sum_{\mu,\nu=1}^n \langle [E_{\mu},F_{\nu}]^{\ast},E_{\mu}^{\ast} \rangle_{\vert\gamma(r(t))} F^{\ast}_{\nu\vert\gamma(r(t))}
\end{gather*}
for $t\in(M/G)^{\circ}$. Here, $E_1,\ldots,E_n \in \mathfrak{n}$ and $F_1,\ldots,F_n \in \mathfrak{n}$ are such that $E^{\ast}_{1\vert\gamma(t)}$,\ldots, $E^{\ast}_{n\vert\gamma(t)}$ form an orthonormal basis of $T_{\gamma(t)}(G\cdot\gamma(t))$ and $F^{\ast}_{1\vert\gamma(r(t))},\ldots, F^{\ast}_{n\vert\gamma(r(t))}$ form an orthonormal basis of $T_{\gamma(r(t))}(G\cdot\gamma(r(t)))$.
\end{Thm}

\section{Harmonic maps between cohomogeneity one manifolds}
\label{sec-har}
In this section we determine the tension field of equivariant harmonic maps between cohomogeneity one manifolds. Unlike \cite{MR4000241}, we do not assume that self-maps are necessarily considered. Nonetheless, the considerations in this section clearly align with those of \cite{MR4000241}, from which we also adopt the notation.

\medskip

Let $G$ and $K$ act on $M$ and $N$ with cohomogeneity one, respectively.
Additionally, let \begin{align*} A: G\rightarrow K \end{align*} denote a group homomorphism. Choose a unit-speed normal geodesic $\gamma$ on $M$ and a unit-speed normal geodesic $\hat{\gamma}$ on $N$.

\smallskip

We define $\psi$ as the equivariant map given by \begin{align*} \psi(g\cdot \gamma(t))=A(g)\cdot \hat{\gamma}(r(t)), \end{align*} where $r:(M/G)^{\circ}\rightarrow \mathbb{R}$ is a smooth function that also meets the boundary conditions specified in cases (3) and (4), see (i) and (ii) in Subsection\,\ref{sub-har}.

We will now calculate the tension field of $\psi$, beginning with the calculation of its normal component.
The normal directional derivative of $\psi$ is expressed as follows: \begin{gather*} d\psi_{\vert \gamma(t)} \cdot \dot \gamma(t) = \frac{d}{dt}(\psi\circ\gamma(t)) = \dot r(t) \dot{\hat{\gamma}}(r(t)). \end{gather*}

Next, we calculate the derivative of $\psi$ along tangential directions.
To achieve this, we utilize the fact that the mapping $$\ag/\ah \to T_p(G\cdot p), X\mapsto X^{\ast}_{\vert p}$$ constitutes a vector space isomorphism at regular points $p$. 
We then obtain the following identity for the derivative of $\psi$ in tangential directions
\begin{gather*}
  d\psi_{\vert\gamma(t)}\cdot X^{\ast}_{\vert \gamma(t)} = \frac{d}{ds} \psi(\exp sX \cdot \gamma(t))_{\vert s= 0}
  = \frac{d}{ds}\bigl(\exp s\, dA_{e}X \cdot \hat{\gamma}(r(t))\bigr)_{\vert s=0} = (dA_{e}X)^{\ast}_{\vert\hat{\gamma}(r(t))}.
\end{gather*}
Here, $e$ is used to denote the identity element of $G$.

\smallskip

The tension field $\tau$ of $\psi$ is defined by (compare the introduction)
\begin{gather}
   \tau_{\vert p} = \sum_{\mu=0}^n \nabla d\psi_{\vert p}(e_{\mu},e_{\mu}) = \sum_{\mu=0}^{n}\Bigl( \nabla_{e_{\mu}} (d\psi\cdot e_{\mu}) - d\psi \cdot \nabla_{e_{\mu}}e_{\mu}\Bigr)_{\vert p},
\label{tensiondef}
\end{gather}
where the vectors $e_0,\ldots,e_n$ form any orthonormal basis of $T_p M$ and can be extended arbitrarily to vector fields on a neighborhood of $p$. 
The equivariance of the tension field implies that it suffices to evaluate the expression along $\gamma(t)$. 
By $T$ we denote the unit normal field to the principal orbits of $M$ given by $T_{\vert g\cdot \gamma(t)} = g\cdot \dot \gamma(t)$.
Moreover, we set $e_0 = \dot\gamma(t)$ and choose $E_1,\ldots, E_n\in \ag$ such that $e_1 = E^{\ast}_{1\vert\gamma(t)},\ldots,e_n = E^{\ast}_{n\vert\gamma(t)}$ form an orthonormal basis of $T_{\gamma(t)}(G\cdot\gamma(t))$. 
We get
\begin{gather*}
  \nabla_{e_0}(d\psi\cdot e_0)_{\vert \gamma(t)} = \frac{\nabla}{dt} \bigl(\dot r(t) \dot{\hat{\gamma}}(r(t))\bigr) = \ddot r(t) \dot{\hat{\gamma}}(r(t))
\end{gather*}
and $(d\psi\cdot \nabla_{e_0} e_0)_{\vert \gamma(t)} = d\psi_{\vert \gamma(t)}\cdot \frac{\nabla}{dt}\dot{\hat{\gamma}}(t) = 0$. Hence, the tension field definition (\ref{tensiondef}) becomes
\begin{gather}
   \tau_{\vert \gamma(t)} = \ddot r(t) \dot{\hat{\gamma}}(r(t)) + \sum_{\mu=1}^{n}\Bigl( \nabla_{(dA_{e}E_{\mu})^{\ast}} (dA_{e}E_{\mu})^{\ast}_{\vert\hat{\gamma}(r(t))}  - d\psi_{\vert \gamma(t)} \cdot \nabla_{E^{\ast}_{\mu}}E^{\ast}_{\mu\vert \gamma(t)}\Bigr).
\label{tensionev}
\end{gather}

In order to present the normal component of the tension field in a computational convenient way, we need to introduce some notation first:
Following \cite{MR4000241}, we let
\begin{gather*}
  \Pi_t^{r(t)}: T_{\hat{\gamma}(t)} (K\cdot\hat{\gamma}(t)) \to T_{\hat{\gamma}(r(t))} (K\cdot \hat{\gamma}(r(t)))
\end{gather*}
denote the parallel transport along the normal geodesic~$\hat{\gamma}$.
Further, we denote by $J_t^{r(t)}$ the endomorphism of
$T_{\hat{\gamma}(t)} (K\cdot\hat{\gamma}(t))$ given by composing the action field homomorphism $X^{\ast}_{\vert\hat{\gamma}(t)} \mapsto X^{\ast}_{\vert \hat{\gamma}(r(t))}$ with $(\Pi_t^{r(t)})^{-1} = \Pi_{r(t)}^t$. 
By $\hat{T}$ we denote the unit normal field to the principal orbits of $N$ given by $\hat{T}_{\vert k\cdot \hat{\gamma}(t)} = k\cdot \dot \hat{\gamma}(t)$, for $k\in K$.
Moreover, we set $f_0 = \dot{\hat{\gamma}}(t)$ and choose $F_1,\ldots, F_m\in \ak$ such that $f_1 = F^{\ast}_{1\vert\hat{\gamma}(t)},\ldots,f_m = F^{\ast}_{m\vert\hat{\gamma}(t)}$ form an orthonormal basis of $T_{\hat{\gamma}(t)}(K\cdot\hat{\gamma}(t))$.

With this preparation at hand we can now provide the normal component of the tension field.

\begin{Thm}
The normal component $\tau^{\nor}_{\vert\gamma(t)} = \langle\tau_{\vert\gamma(t)}, \dot{\hat{\gamma}}(r(t))\rangle$ of the tension field is given by
\begin{gather*}
  \tau^{\nor}_{\vert\gamma(t)} = \ddot r(t) - \dot r(t) \tr S_{\vert \gamma(t)}
    + \sum_{\mu=1}^n\sum_{i,j=1}^mc_{\mu,i}c_{\mu,j}\langle(J_t^{r(t)})^{\ast} (\Pi_t^{r(t)})^{-1} \hat{S}_{\vert \hat{\gamma}(r(t))} \Pi_t^{r(t)} J_t^{r(t)} f_i,f_j\rangle_{\lvert\hat{\gamma}(t)},
\end{gather*}
where
\begin{gather*}
c_{\mu,i}:=c_{\mu,i}(t):=\langle (dA_{e}E_{\mu})^{\ast}, f_i\rangle_{\lvert\hat{\gamma}(t)}.
\end{gather*}
Moreover, $S_{\vert \gamma(t)}$ denotes the shape operator of the orbit $G\cdot \gamma(t)$ at $\gamma(t)$, $\hat{S}_{\vert \hat{\gamma}(t)}$ denotes the shape operator of the orbit $K\cdot \hat{\gamma}(t)$ at $\hat{\gamma}(t)$ and $(J_t^{r(t)})^{\ast}$ denotes the adjoint endomorphism of $J_t^{r(t)}$.
\label{normalone}
\end{Thm}
\begin{proof}
We evaluate the tension field formula (\ref{tensionev})
and make use of the identity
\begin{gather*}
  \langle \nabla_{X^{\ast}} Y^{\ast}, T \rangle =
    \langle X^{\ast}, S \cdot Y^{\ast}\rangle,
\end{gather*}
which holds
by the definition of the shape operator. 
Consequently,
\begin{gather*}
  \langle \sum_{\mu=1}^n \nabla_{E^{\ast}_{\mu}} E^{\ast}_{\mu}, T \rangle_{\vert\gamma(t)} =
  \sum_{\mu=1}^n \langle e_{\mu}, S_{\vert\gamma(t)} e_{\mu}\rangle = \tr S_{\vert\gamma(t)}.
\end{gather*}
Furthermore, we obtain 
\begin{multline*}
\langle \sum_{\mu=1}^{n}\nabla_{(dA_{e}E_{\mu})^{\ast}} (dA_{e}E_{\mu})^{\ast}, \hat{T}\rangle_{\vert\hat{\gamma}(r(t))} 
 =
  \sum_{\mu=1}^n \langle (dA_{e}E_{\mu})^{\ast}_{\vert\hat{\gamma}(r(t))}, \hat{S}_{\vert \hat{\gamma}(r(t))} (dA_{e}E_{\mu})^{\ast}_{\vert\hat{\gamma}(r(t))} \rangle\\
  = \sum_{\mu=1}^n \langle J_t^{r(t)} (dA_{e}E_{\mu})^{\ast}_{\vert\hat{\gamma}(t)}, (\Pi_t^{r(t)})^{-1} \hat{S}_{\vert\hat{\gamma}(r(t))} \Pi_t^{r(t)} J_t^{r(t)} (dA_{e}E_{\mu})^{\ast}_{\vert\hat{\gamma}(t)}\rangle.
\end{multline*}
Plugging  
\begin{align*}
 (dA_{e}E_{\mu})^{\ast}_{\vert\hat{\gamma}(t)}=\sum_{i=1}^mc_{\mu,i}{F_i^{\ast}}_{\vert\hat{\gamma}(t)}   
\end{align*}
into the previous equation yields the claim.
\end{proof}

\begin{Bem}
Theorem\,\ref{normalone} generalizes Theorem\,A in \cite{MR4000241} and recovers Theorem\,A when $M=N$, $G=K$ and $A=\Id$.
\end{Bem}

As in \cite{MR4000241} we provide a second formula for the normal component of the tension field, namely a formula in terms of data of the acting groups $G$ and $K$. 
Let $\hat{Q}$ be a fixed biinvariant metric on $K$. Denote the orthonormal complement of the Lie algebra of the principal isotropy group in $\ak$ by $\hat{\mathfrak{n}}$. Define the metric endomorphisms $\hat{P}_t : \hat{\mathfrak{n}} \to \hat{\mathfrak{n}}$ by
\begin{gather*}
  \hat{Q}(X, \hat{P}_t\cdot Y) = \langle X^{\ast},Y^{\ast} \rangle_{\vert\hat{\gamma}(t)}.
\end{gather*}

We have
\begin{multline*}
  \langle X^{\ast},\hat{S}\cdot X^{\ast}\rangle_{\vert\hat{\gamma}(r(t))}=  -\langle X^{\ast}, \nabla_{\hat{T}} X^{\ast}\rangle_{\vert\hat{\gamma}(r(t))}= -\tfrac{1}{2\dot r(t)} \tfrac{d}{dt} \langle X^{\ast},X^{\ast}\rangle_{\vert\hat{\gamma}(r(t))} \\
  = -\tfrac{1}{2} \hat{Q}(X,(\dot{\hat{P}})_{r(t)} X)
  = -\tfrac{1}{2} \langle X^{\ast}, (\hat{P}_t^{-1}(\dot{ \hat{P}})_{r(t)} X)^{\ast} \rangle_{\vert\hat{\gamma}(t)}
\end{multline*}
and hence the following statement holds.

\begin{Thm}
The normal component of the tension field is given by
\begin{gather*}
  \tau^{\nor}_{\vert\gamma(t)} = \ddot r(t) + \tfrac{1}{2}\dot r(t) \tr P_t^{-1}\dot P_t
    - \tfrac{1}{2}\sum_{\mu=1}^n\sum_{i,j=1}^mc_{\mu,i}c_{\mu,j}P_{i,j},
\end{gather*}
where \begin{gather*}
c_{\mu,i}:=c_{\mu,i}(t):=\langle (dA_{e}E_{\mu})^{\ast}, f_i\rangle_{\lvert\hat{\gamma}(t)}.
\end{gather*}
and $P_{i,j}$ is the $(i,j)$-entry of the representing matrix of $\hat{P}_t^{-1}(\dot{ \hat{P}})_{r(t)}$.
\label{normaltwo}
\end{Thm}

\begin{Bem}
Theorem\,\ref{normaltwo} generalizes Theorem\,3.4 in \cite{MR4000241} and recovers Theorem\,3.4 when $M=N$, $G=K$ and $A=\Id$.
\end{Bem}

The computation of the tangential component of the tension field 
is analogous to the computation in \cite{MR4000241}[Theorem\,3.6] and is therefore omitted. We obtain:

\begin{Thm}
\label{tanaltpart}
The tangential component of the tension field is given by
\begin{gather*}
  \tau^{\tan}_{\vert\gamma(t)} = 
  \bigl(\hat{P}_{r(t)}^{-1} \sum_{\mu=1}^n [dA_{e}E_{\mu},\hat{P}_{r(t)}dA_{e}E_{\mu}]\bigr)^{\ast}_{\vert\hat{\gamma}(r(t))}
\end{gather*}
where $E_1,\ldots,E_m \in \mathfrak{n}$ are such that $E^{\ast}_{1\vert\gamma(t)},\ldots, E^{\ast}_{n\vert\gamma(t)}$ form an orthonormal basis of $T_{\gamma(t)}(G\cdot{\gamma}(t))$.
\end{Thm}

We apply the preceding results to study conformal changes of the metric $g_t$ of the domain manifold.

\begin{Lem}
Let $(M,g)$ be a cohomogeneity one manifold.
Further, let $\psi:(M,dt^2+\exp(2\alpha(t))g_t)\rightarrow (M,g)$ be the equivariant map given by 
\begin{align*}
\psi(g\cdot \gamma(t))=g\cdot {\gamma}(r(t)), g\in G,
\end{align*}
where $r:(M/G)^{\circ}\rightarrow \mathbb{R}$ which in addition satisfies boundary conditions in cases (3) and (4).
Then we have
\begin{gather*}
  \tau^{\nor}_{\vert\gamma(t)} = \ddot r(t) + \tfrac{1}{2}\dot r(t) \tr P_t^{-1}\dot P_t+\dim(\mathfrak{n})\dot \alpha(t)\dot r(t)
    - \tfrac{1}{2} \tr P_t^{-1} (\dot P)_{r(t)}.
\end{gather*}
\end{Lem}
\begin{proof}
This follows from Theorem\,\ref{normaltwo}.
\end{proof}

\section{The initial value problem for harmonic maps}
\label{sec-ivp}

This section is dedicated to establishing and solving the initial value problem for harmonic maps of cohomogeneity-one manifolds.
We will assume the existence of a singular orbit \( N \); if not, the initial value problem is a regular ODE which meet specific initial values and therefore has a solution.
Thus, $M/G$ is either isometrically equivalent to $[0,1]$ or to $[0,\infty)$. We establish the initial value problem at $t=0$ below.

\medskip

The main result of this section is:

\begin{Thm}
   \label{main} 
   Let $(M,g)$ be a cohomogeneity one manifold with singular orbit, i.e. $M/G$ is isometric to a closed interval $[0,L]$ or a half line $[0,\infty)$.
For any $v\in\mathbb{R}$ the differential equation
\begin{gather*}
  0 = \ddot r(t) - \dot r(t) \tr S_{\vert \gamma(t)}
    + \tr\, (J_t^{r(t)})^{\ast} (\Pi_t^{r(t)})^{-1} S_{\vert \gamma(r(t))} \Pi_t^{r(t)} J_t^{r(t)}
\end{gather*}
has a smooth solution $r:(M/G)^{\circ}\cup\{0\} \to \mathbb{R}$ 
with $\lim_{t\to 0} r(t) = 0$ and  $\lim_{t\to 0} \dot r(t) = v$, which depends continuously on $v$.
\end{Thm}

Prior to demonstrating Theorem\,\ref{main}, we note the following:
We examine metrics that take the shape of (\ref{metric}) and have established \begin{align*} g(X^{\ast}, Y^{\ast})_{\gamma(t)}=g_t(X,Y)=Q(P_tX, Y) \end{align*} for $X, Y\in\mathfrak{n}$. 
It should be noted that the endomorphism \( P_t:\mathfrak{n}\rightarrow\mathfrak{n} \) possesses the properties of being \( Q \)-symmetric and \( \Adj_H \)-equivariant.
Regarding the metric $Q$, we examine the following decomposition of the Lie algebra $\mathfrak{g}$: \begin{align*} \mathfrak{g}=\mathfrak{h}\oplus\mathfrak{p}\oplus\mathfrak{m}. \end{align*}
Here, $\mathfrak{h}\oplus\mathfrak{p}$ represents the Lie algebra of $K$, denoted as $\mathfrak{k}$.
Further note that \begin{align*} \mathfrak{n}=\mathfrak{p}\oplus\mathfrak{m}. \end{align*}
Keep in mind that $H$ operates on $\mathfrak{n}$ through the adjoint representation.
Additionally, a $G$-invariant metric on $G/H$ can be characterized by an inner product on $\mathfrak{n}$ that is invariant under $\Adj_H$.\footnote{The notation and facts summarized in the previous paragraph have been used in various manuscripts, see e.g. \cite{MR1923478, VZ1, MR4400726, MR4000241}.}

\smallskip

We are now ready to prove Theorem\,\ref{main}.

\begin{proof}[Proof of Theorem\,\ref{main}]
In the first step of the proof, the initial value problem is determined.

\smallskip

As in the manuscripts \cite{VZ1,MR1758585} we first make the following assumption: 
\begin{Ass} \label{ass}
  In $\mathfrak{m}$, there are no irreducible representations of $\Adj_H$ that correspond to those in $\mathfrak{p}$. \end{Ass}
 We will eliminate this assumption later in the proof.
 Its purpose is to facilitate the calculations and thus clarify the proof's strategy.

The space of all symmetric tensors of order $2$ on a vector space $V$ is denoted by $S^2(V)$.
We denote by $S^2(V)^H$ the space of all $H$-invariant symmetric tensors of order $2$ on a vector space $V$. Due to assumption\,\ref{ass}, we have the following splitting: \begin{align*} S^2(\mathfrak{n})^H= S^2(\mathfrak{p}\oplus\mathfrak{m})^H = S^2(\mathfrak{p})^H\oplus S^2(\mathfrak{m})^H, \end{align*} as noted in \cite{MR1758585}.

The normal component of the tension field will be established with the help of Theorem\,\ref{tau-nor}.

According to (3.2) in \cite{VZ1}, we have \begin{align*} &{P_t}_{\lvert\mathfrak{p}}=t^2\Id_{\mathfrak{p}}+t^4 B(t),\\ &{P_t}_{\lvert\mathfrak{m}}=A_0+tA_1+t^2A(t), \end{align*} where $A_0,A_1$ are constant matrices and $A, B$ are smooth functions of $t$.
Keep in mind that ${P_t}_{\lvert\mathfrak{mp}}=0$, as a result of assumption (\ref{ass}).
Thus we obtain \begin{align*} &\dot{P_t}_{\lvert\mathfrak{p}}=2t\Id_{\mathfrak{p}}+4t^3B(t)+t^4 \dot B(t),\\ &\dot {P_t}_{\lvert\mathfrak{m}}=A_1+2tA(t)+t^2\dot A(t).
\end{align*}

From (3.3) in \cite{VZ1} we have
\begin{align*}
    &{P_t}^{-1}_{\lvert\mathfrak{p}}=\frac{1}{t^2}\Id_{\mathfrak{p}}-B(t)+...,\\
     &{P_t}^{-1}_{\lvert\mathfrak{m}}=A_0^{-1}-tA_0^{-1}A_1A_0^{-1}+... \quad.
\end{align*}
Thus we get
\begin{align*}
  &\frac{1}{2}\tr _{\mathfrak{p}}({P_t}^{-1}\dot{P_t})=\frac{1}{t}\dim\mathfrak{p}+f_1(t),\\
  &\frac{1}{2}\tr _{\mathfrak{m}}({P_t}^{-1}\dot{P_t})=f_2(t),\\
\end{align*}
where $f_1,f_2:[0,\infty)\rightarrow\mathbb{R}$ are smooth functions. 
Hence we obtain
\begin{align*}
  &\frac{1}{2}\tr _{\mathfrak{p}}({P_t}^{-1}\dot{P}_{r(t)})=\frac{r(t)}{t^2}\dim\mathfrak{p}+\frac{r^3(t)}{t^2}h_1(t,r(t))+\frac{r^4(t)}{t^2}h_2(t,r(t))+f_3(t,r(t)),\\
  &\frac{1}{2}\tr _{\mathfrak{m}}({P_t}^{-1}\dot{P}_{r(t)})=f_4(t,r(t)),\\
\end{align*}
where $h_1,h_2,f_3,f_4:[0,\infty)\times\mathbb{R}\rightarrow\mathbb{R}$ are smooth functions. 

We thus derive from Theorem\,\ref{tau-nor}:
\begin{align*} \ddot r(t)+\dim\mathfrak{p}\frac{\dot r(t)}{t}-\dim\mathfrak{p}\frac{r(t)}{t^2}+\frac{r^3(t)}{t^2}h_1(t, r(t))+\frac{r^4(t)}{t^2}h_2(t, r(t))+f_5(t,r(t),\dot r(t))=0, \end{align*}
 where $$f_5(t,r(t),\dot r(t))=(f_1(t)+f_2(t))\dot r(t)+f_3(t,r(t))+f_4(t,r(t)).$$ It is evident that $f_5:[0,\infty)\times\mathbb{R}^2\rightarrow\mathbb{R}$ is smooth.

\medskip

We will address the general case now, meaning we no longer assume that (\ref{ass}) holds.
Keep in mind that for a block matrix \begin{align*} X=\begin{pmatrix} A & B\\ C & D \end{pmatrix}, \end{align*} the matrices $A$ and $D$ are square.
If $A$ and its Schur complement $$E:=(D-CA^{-1}B)^{-1}$$ are invertible, then $X$ is invertible and the inverse of $X$ is given by \begin{align*} X^{-1}=\begin{pmatrix}
A^{-1}+A^{-1}BECA^{-1} & -A^{-1}BE\\ -ECA^{-1} & E \end{pmatrix}. \end{align*}
(An analogous statement holds if $D$ and its Schur complement $A-BD^{-1}C$ are invertible.)
We utilize these findings to \begin{align*}
P_t=\begin{pmatrix}
{P_t}_{\lvert\mathfrak{p}} & {P_t}_{\lvert\mathfrak{mp}}\\ {P_t}_{\lvert\mathfrak{mp}}^{\T} & {P_t}_{\lvert\mathfrak{m}} \end{pmatrix}. \end{align*}

From (3.2) and (3.3) in \cite{VZ1} we have
\begin{align*}
    &{P_t}_{\lvert\mathfrak{mp}}=t^2C_0+t^3 C(t),\\
     &{P_t}_{\lvert\mathfrak{mp}}^{-1}=-A_0^{-1}C_0+...,
\end{align*}
where $C_0$ is a constant matrix and $C$ is smooth.
We start by determining the Schur complement $E$ of ${P_t}_{\lvert\mathfrak{p}}$.
Plugging in the expressions for ${P_t}_{\lvert\mathfrak{mp}}, {P_t}_{\lvert\mathfrak{m}}, {P_t}_{\lvert\mathfrak{p}}$ yields
\begin{align*}
    E=({P_t}_{\lvert\mathfrak{m}}-{P_t}_{\lvert\mathfrak{mp}}^{\T}{P_t}_{\lvert\mathfrak{p}}{P_t}_{\lvert\mathfrak{mp}})^{-1}=A_0^{-1}-tA_0^{-1}A_1A_0^{-1}+... \quad.
\end{align*}
Analogous calculations yield
\begin{align*}
P_t^{-1}=\begin{pmatrix}
\frac{1}{t^2}\Id_\mathfrak{p}+... & -C_0A_0^{-1}+...\\
-C_0^{\T}A_0^{-1}+... & A_0^{-1}-tA_0^{-1}A_1A_0^{-1}+...
\end{pmatrix}.    
\end{align*}
Thus straightforward computations yield
\begin{align*}
  &\frac{1}{2}\tr _{\mathfrak{p}}({P_t}^{-1}\dot{P_t})=\frac{1}{t}\dim\mathfrak{p}+f_1(t),\\
  &\frac{1}{2}\tr _{\mathfrak{m}}({P_t}^{-1}\dot{P_t})=f_2(t),\\
\end{align*}
where $f_1,f_2:[0,\infty)\rightarrow\mathbb{R}$ are smooth functions. 
Further, we obtain
\begin{align*}
  &\frac{1}{2}\tr _{\mathfrak{p}}({P_t}^{-1}\dot{P}_{r(t)})=\frac{r(t)}{t^2}\dim\mathfrak{p}+\frac{r^3(t)}{t^2}h_1(t,r(t))+\frac{r^4(t)}{t^2}h_2(t,r(t))+f_3(t,r(t)),\\
  &\frac{1}{2}\tr _{\mathfrak{m}}({P_t}^{-1}\dot{P}_{r(t)})=f_4(t,r(t)),\\
\end{align*}
where $f_3,f_4:[0,\infty)\times\mathbb{R}\rightarrow\mathbb{R}$ and $h_1,h_2:[0,\infty)\times\mathbb{R}\rightarrow\mathbb{R}$ are smooth functions. Note that the functions $f_1, f_2, f_3, f_4$ are possibly not the same as above, where we assumed (\ref{ass}).
Thus, as above, we get
from Theorem\,\ref{tau-nor}:
\begin{align}
\label{dgl}
\ddot r(t)+\dim\mathfrak{p}\frac{\dot r(t)}{t}-\dim\mathfrak{p}\frac{r(t)}{t^2}+\frac{r^3(t)}{t^2}h_1(t,r(t))+\frac{r^4(t)}{t^2}h_2(t,r(t))+f_5(t,r(t),\dot r(t))=0,
\end{align}
 where $$f_5(t,r(t),\dot r(t))=(f_1(t)+f_2(t))\dot r(t)+f_3(t,r(t))+f_4(t,r(t))$$
and $f_5:[0,\infty)\times\mathbb{R}^2\rightarrow\mathbb{R}$ is smooth.
Thus, we need to solve the initial value problem 
\begin{align*}
\ddot r(t)+\dim\mathfrak{p}\frac{\dot r(t)}{t}-\dim\mathfrak{p}\frac{r(t)}{t^2}+\frac{r^3(t)}{t^2}h_1(t,r(t))+\frac{r^4(t)}{t^2}h_2(t,r(t))+f_5(t,r(t),\dot r(t))=0,
\end{align*}
with $r(0)=0$, $\dot r(t)=v$, where $v\in\mathbb{R}$.

\smallskip

To solve this initial value problem, we employ a modified Picard iteration approach.
The main steps of the proof are provided here; this method has been described in detail in \cite{MR800005,MR1758585}.

\smallskip

In a first step, it is demonstrated that the initial conditions $r(0)=0$ and $r'(0)=v$ uniquely determine the initial derivatives of any order.
In order to prove this we make a power series Ansatz
\begin{align}
    \label{psa}
    r(t)=\sum_{i=1}^{\infty}c_it^i,
\end{align}
with $c_i\in\mathbb{R}$ and $c_1=v.$
Assume that we have proven that all $c_{i}$ with $i\leq i_0$ are uniquely determined, where $i_0\in\mathbb{N}$ and $i_0\geq 2$.
Then $c_{i_0+1}$ is also uniquely determined:
Plugging (\ref{psa}) into (\ref{dgl}) yields
\begin{align*}
(c_{i_0+1}((i_0+1)i_0+\dim\mathfrak{p}\,i_0) +f(c_1,\dots,c_{i_0})  )t^{i_0-1}+O(t^{i_0})=0,
\end{align*}
where $f:\mathbb{R}^{i_0}\rightarrow\mathbb{R}$ is a smooth function depending only on $c_1,\dots,c_{i_0}$ which are uniquely determined by assumption. Hence $c_{i_0+1}$ is the unique solution of the equation
\begin{align*}
    c_{i_0+1}(i_0+1+\dim\mathfrak{p})i_0 +f(c_1,\dots,c_{i_0})  =0.
\end{align*}
Hence
 the initial conditions $r(0)=0$ and $r'(0)=v$ uniquely determine the initial derivatives of any order.

\smallskip

We can rewrite (\ref{dgl}) as a system of first order equations:
\begin{align}
\label{system-dgl}
    \frac{d}{dt}\begin{pmatrix}
r(t) \\
u(t) 
\end{pmatrix}=\begin{pmatrix}
u(t) \\
\frac{1}{t^2}h_3(r(t))+\frac{1}{t}h_4(t,r(t),u(t)) 
\end{pmatrix}
\end{align}
where $h_3:\mathbb{R}\rightarrow\mathbb{R}$ and $h_4:[0,\infty)\times\mathbb{R}^2\rightarrow\mathbb{R}$ are smooth functions.
Note that $f_5$ was merged into $h_4$ which therefore now depends explicitly on $t$.
Note further that $h_3$ and $h_4$ are bounded and uniformly Lipschitz in $r$ and $u$ provided that $r$ and $u-v$ are bounded by $2\epsilon$, $\epsilon$ being a small positive constant.
We will make use of this fact later.

\smallskip

Let $t_0>0$ be fixed.
We want to solve the initial value problem of the system (\ref{system-dgl}) with initial values $r(0)=0, u(0)=v$, where $v\in\mathbb{R}$.
We rewrite (\ref{system-dgl}) as a system of integral equations
\begin{align*}
    &r(t)=\int_{0}^tu(s) ds,\\
    & u(t)=v+\int_{0}^t(\frac{1}{s^2}h_3(r)+\frac{1}{s}h_4(s,r,u) )ds.
\end{align*}
Let $u_n, r_n$ be the $n$-th order Taylor polynomial of a prospective solution of the initial value problem, see \cite{MR800005}, i.e. $r_n$  and $u_n$ are of order $n+1$ and $n$ respectively, and satisfy $r_n(0)=0$, $u_n(0)=v$ as well as
\begin{align}
\label{system-dgl}
    \frac{d}{dt}\begin{pmatrix}
r_n(t) \\
u_n(t) 
\end{pmatrix}=\begin{pmatrix}
u_n(t)+\mathcal{O}(n+1) \\
\frac{1}{t^2}h_3(r_n(t))+\frac{1}{t}h_4(t,r_n(t),u_n(t))+ \mathcal{O}(n).
\end{pmatrix}
\end{align}
Here $\mathcal{O}(p)=\{(\phi,\psi):[0,t_0]\rightarrow\mathbb{R}^2 \,\mbox{smooth}\,\lvert\,\phi^{(i)}(0)=\psi^{(i)}(0)=0\,\mbox{for}\, 1\leq i\leq p\}$.
We endow $\mathcal{O}(p)$ with the norm
\begin{align*}
    \norm{f}_p=\mbox{sup}_{[0,t_0]}\frac{\lvert f\rvert}{t^p}
\end{align*}
and thus obtain a Banach space, see \cite{MR800005,MR1758585}.
Let
\begin{align*}
    \xi=r-r_n, \eta=u-u_n.
\end{align*}
We finish the proof by showing that the operator
$$\mathcal{L}=(\mathcal{L}_1, \mathcal{L}_2)$$ defined by
\begin{align*}
& \mathcal{L}_1(\xi,\eta)=-r_n(t)+\int_0^{t}(u_n(t)+\eta)(s)ds,\\   & \mathcal{L}_2(\xi,\eta)=v-u_n(t)+\int_0^{t}(\frac{1}{s^2}h_3(s, r_n+\xi)+\frac{1}{s}h_4(s,r_n+\xi,u_n+\eta))(s)ds
\end{align*}
is a contraction on a closed ball $B_{\epsilon}(0)\subset\mathcal{B}=\mathcal{O}(n+1)\times \mathcal{O}(p)$, where $\mathcal{B}$ is endowed with the norm $\norm{(\xi, \eta)}=\norm{\xi}_{n+1}+\norm{\eta}_n$.
For $t_0$ small enough, we obtain form the fact that $h_3, h_4$ are Lipschitz and a few easy considerations that
\begin{align*}
    \norm{\mathcal{L}_1(\xi_1,\eta_1)-\mathcal{L}_1(\xi_2,\eta_2)}
\leq \frac{c_1}{n+1}\norm{(\xi_1,\eta_1)-(\xi_2,\eta_2)},\\
   \norm{\mathcal{L}_2(\xi_1,\eta_1)-\mathcal{L}_2(\xi_2,\eta_2)}
\leq \frac{c_2}{n}\norm{(\xi_1,\eta_1)-(\xi_2,\eta_2)},
\end{align*}
where $c_2$ does not depend on $n$ (this follows from the fact that $h_3$ and $h_4$ are uniformly Lipschitz). For $n$ large enough $\mathcal{L}$ is therefore a contraction on $B_{\epsilon}(0)$ and thus we get a fix point of $\mathcal{L}$ which provides a solution of the initial value problem
\begin{align*}
    \frac{d}{dt}\begin{pmatrix}
r(t) \\
u(t) 
\end{pmatrix}=\begin{pmatrix}
u(t) \\
\frac{1}{t^2}h_3(r(t))+\frac{1}{t}h_4(t,r(t),u(t)) 
\end{pmatrix}
\end{align*}
$r(0)=0$, $u(0)=v$, whence the claim.
\end{proof}

\begin{Bem}
\begin{enumerate}
\item For specific cases the initial value problem
for harmonic maps of cohomogeneity one manifolds has been solved in the literature, see e.g. \cite{MR1298998, MR2022387,MR3427685, MR3745872}.
\item The results of Theorem\,\ref{main} can be applied to solve the initial value problem for harmonic maps between so-called models, see \cite{MR3357596}, in particular equation (2.7), for details.
\item The same strategy as in the present section can be used to solve the initial value problem for equivariant harmonic maps between (possibly) different cohomogeneity one manifolds.

\end{enumerate}
\end{Bem}

\section{Regular-singular systems of first order}
\label{sec4}
In the present section we present a theory for regular-singular systems of first order.
Many of the results are known, but scattered throughout the literature. 
The organization of the section is as follows: In the first two subsections we consider
the regular-singular differential equation/system whose solutions are power functions/vector-valued generalizations of power functions.
Basic phenomena are explained with the help of these examples.
Afterwards, in Subsections\,\ref{sub3}-\ref{sub7} we study linear regular-singular system of first order.
In particular, we discuss the monodromy map. Finally, in Subsections\,\ref{sub8}-\ref{sub10} we study non-linear regular-singular system of first order.

We refer to \cite{MR800005,MR4269412,MR2682403, MR2975151,MR2961944,mon,MR1758585} and the references therein.

\subsection{Scalar power functions}
\label{sub1}

The \textit{power function} $y=c\cdot x^a$, with $a\in\C$, $c\in\C^{*}$, satisfies the differential equation

\begin{align}
\label{raute1}
\frac{dy}{dx}=\frac{a}{x}\cdot y,
\end{align}

and is, up to a scaling factor, determined uniquely by this differential equation.
Notice, that the power function $y$ is a holomorphic function in the classical sense if and only if $a\in\mathbb{Z}$.
For all other choices of $a\in\mathbb{C}$, one has to slice the complex plane along one ray emanating from the origin, for example one could slice the complex plane along the negative real axis.
In other words, we have to deal with a multi-valued function, which is better characterized by its graph
\begin{align*}
\left\{(x,y)\in\C^{*}\times\C\,\lvert\,y=c\cdot x^a \right\}=\mbox{im}\big( z\mapsto (e^z, c\cdot e^{az}), z\in\C\big).
\end{align*}

\begin{Bem}
The fact that the right hand side of (\ref{raute1}) has only one pole of first order in $x=0$ is crucial for the identification of $y$ with a power function.
Notice, that already slightly more complicated singularities on the right hand side of (\ref{raute1}) may cause decisively more complicated essential singularities.
For example, the functions

\begin{equation*}
y=c\cdot e^{-1/x},\hspace{0.1cm}x\in\C^{*}
\end{equation*}

satisfy the differential equation

\begin{align*}
\frac{dy}{dx}=\frac{1}{x^2}\cdot y,
\end{align*}

whose right hand side has a pole of second order at $x=0$.
\end{Bem}

\subsection{Vector-valued generalizations of the scalar power functions}
\label{sub2}
By the considerations of the preceding subsection it seems natural to consider the holomorphic maps
$Y:U\rightarrow\C^{n}$, with open $U\subset\C^{*}$, which satisfy
\begin{align}
\label{raute2}
\frac{d}{ds}Y=\frac{1}{s}A_0\cdot Y,
\end{align}
for a fixed $A_0\in\mbox{End}(\C^n)$, as vector-valued generalizations of the scalar power functions.

\smallskip

As in the case of (\ref{raute1}), the solutions of the differential equation (\ref{raute2})
can be determined explicitly. Indeed, using the projection $\widehat{\pi}:\C\rightarrow\C^{*}$, $z\mapsto s=e^z$,
one can pull back the differential equation to a simply-connected domain of definition.
Identity (\ref{raute2}) thus transforms into

\begin{align*}
\frac{d}{dz}(Y\circ\widehat{\pi})=A_0\cdot\big(Y\circ\widehat{\pi}\big)
\end{align*}

and consequently one gets

\begin{align}
\label{raute3}
(Y\circ\widehat{\pi})(z)=\exp(z\, A_0)\cdot(Y\circ\widehat{\pi})(0),
\end{align}

or, less formal,

\begin{align*}
Y(s)=\exp(\ln(s)\, A_0)\cdot Y(1).
\end{align*}

Notice, that $Y:\C^{*}\rightarrow\C^n$ is a holomorphic function in the classical sense if and only if $A_0$ is diagonalizable and has only integer eigenvalues.

\smallskip

Otherwise, one has to deal with \textit{non-trivial monodromy}.
If one solves (\ref{raute2}) along a circle line $c:\theta\mapsto s_0\,e^{i\theta}$, $0\leq\theta\leq 2\pi$, 
then
\begin{align*}
\big(Y\circ c\big)(2\pi)\neq \big(Y\circ c\big)(0)
\end{align*}
holds in the general case. Indeed,
using equation (\ref{raute3}) and choosing $z_0\in\C$ such that $s_0=e^{z_0}$, we get
\begin{align*}
\big(Y\circ c\big)(2\pi)=\exp(2\pi i\, A_0)\cdot\big(Y\circ c\big)(0).
\end{align*}
One calls $M_0:=\exp(2\pi i\, A_0)$ the \textit{monodromy generator}.

\subsection{Linear regular-singular system of first order in $n$ variables}
\label{sub3}
One can substitute (\ref{raute2}) by a more general differential equation without changing the solution theory decisively.
We consider domains of the form 

$$\Omega_{\rho}=\left\{\begin{array}{ll} B(0,\rho)\subset\C,& \mbox{if $0<\rho<\infty$;} \\
         \C,&\mbox{if $\rho=\infty$;}
         \end{array}\right.$$
and set $\Omega_{\rho}^{\circ}=\Omega_{\rho}-\left\{0\right\}$ for $0<\rho<\infty$.

\begin{Dfn}
Let $A:\Omega_{\rho}\rightarrow\mbox{End}(\C^n)$ and $h:\Omega_{\rho}\rightarrow\C^n$ be holomorphic maps.
Then the differential equation system
\begin{align}
\label{raute5}
\frac{d}{ds}Y+\frac{1}{s}A(s)\cdot Y=h(s),\hspace{0.2cm} s\in\Omega_{\rho}^{\circ},
\end{align}
is called a \textit{(local) linear regular-singular system of first order in $n$ variables}.
\end{Dfn}

The point $s=0$ is henceforth referred to as \textit{singular point}. As usual, if $h=0$ the system (\ref{raute5}) is called \textit{homogeneous} and  
\textit{inhomogeneous} otherwise.

\smallskip

Here, too it is useful to consider the pull-back of (\ref{raute5}) with respect to the covering map $\widehat{\pi}:\C\rightarrow\C^{*}$, $z\mapsto s=e^{z}$.
In other words, we deal with the differential equation system
\begin{align}
\label{raute6}
\frac{d}{dz}\widehat{Y}+\widehat{A}(z)\cdot\widehat{Y}=e^z\,\widehat{h}(z),\, z\in\widehat{\pi}^{-1}(\Omega_{\rho}^{\circ}),
\end{align}
with $\widehat{A}(z)=A\circ\pi(z)=A(e^z)$ and $\widehat{h}(z)=h\circ\widehat{\pi}(z)=h(e^z)$.
Since the set $$\widehat{\pi}^{-1}(\Omega_{\rho}^{\circ})=\left\{z\in\C\,\lvert\,\Re(z)<\ln(\rho)\right\}$$ is a half-space, it is simply connected.
Consequently, there exist global solutions $\widehat{Y}:\widehat{\pi}^{-1}(\Omega_{\rho}^{\circ})\rightarrow\C^n$ of (\ref{raute6}).
Furthermore, one can prescribe the initial value
\begin{align}
\label{raute6'}
\widehat{Y}(z_0)=\widehat{Y_0}\in\C^n
\end{align}
at the arbitrarily chosen basis point $z_0\in\widehat{\pi}^{-1}(\Omega_{\rho}^{\circ})$.
The solution of the initial value problem (\ref{raute6}), (\ref{raute6'}) is determined uniquely.

\subsection{Fundamental solution of the homogeneous system}
For an explicit characterization of the solutions of (\ref{raute6}),(\ref{raute6'}) it is useful to introduce
the fundamental solution $\widehat{U}:\widehat{\pi}^{-1}(\Omega_{\rho}^{\circ})\rightarrow\mbox{End}(\C^n)$
of the associated homogeneous system
\begin{align}
\label{raute6hom}
\frac{d}{dz}\widehat{U}+\widehat{A}(z)\cdot\widehat{U}=0,\hspace{1cm}\widehat{U}(z_0)=\eins.
\end{align}
It is well-known that $U(z)\in\mbox{Gl}(n,\C)\subset\mbox{End}(\C^n)$ for all $z\in\widehat{\pi}^{-1}(\Omega_{\rho}^{\circ})$.

\smallskip

Since $\widehat{A}(z+2\pi i)=\widehat{A}(z)$, for any solution $U$ of (\ref{raute6}), (\ref{raute6'}) the map

\begin{align*}
z\mapsto\widehat{U}(z+2\pi i)\cdot \widehat{U}(z_0+2\pi i)^{-1}
\end{align*}

is a solution of the initial value problem (\ref{raute6}), (\ref{raute6'}). Consequently, we obtain

\begin{align}
\label{raute7}
\widehat{U}(z+2\pi i)=\widehat{U}(z)\cdot \widehat{U}(z_0+2\pi i).
\end{align}

Making use of $\widehat{U}(z_0)=\eins$ we thus get

\begin{align}
\label{raute7'}
\widehat{U}(z)^{-1}\widehat{U}(z+2\pi i)=\widehat{U}(z_0)^{-1}\widehat{U}(z_0+2\pi i).
\end{align}

In the general case the right hand side of the preceding equation does not coincide with the identity matrix. Consequently, in general we do not have
a holomorphic map $U:\Omega_{\rho}^{\circ}\rightarrow\mbox{Gl}(n,\C)$ with $\widehat{U}(z)=U(e^z)$. However, the following lemma holds.

\begin{Lem}
There exists a holomorphic map $M:\Omega_{\rho}^{\circ}\rightarrow\mbox{Gl}(n,\C)$, such that the identity $\widehat{U}(z+2\pi i)\cdot \widehat{U}(z)^{-1}=M(e^z)$ holds.
\end{Lem}

Obviously, the map $M$ is independent of the choice of the initial values in (\ref{raute6hom}). Important is only that $\widehat{U}(z_0)\in\mbox{Gl}(n,\C)$
holds. The significance of the \textit{monodromy map} $M$ is explained later on.

\begin{proof}
Using (\ref{raute7}) one obtains by a straightforward computation that $\widehat{M}(z)=\widehat{U}(z+2\pi i)\cdot\widehat{U}(z)^{-1}$
satisfies $\widehat{M}(z+2\pi i)=\widehat{M}(z)$.
\end{proof}

\subsection{Solutions of the inhomogeneous equation}

In the present subsection we deal with the solutions of the inhomogeneous equation (\ref{raute6}).
If $\widehat{h}=0$, then the map

\begin{align*}
z\mapsto\widehat{Y}(z):=\widehat{U}(z)\cdot \widehat{Y_0}
\end{align*}

is a solution of (\ref{raute6}) with the correct initial value $\widehat{Y}(z_0)=\widehat{Y_0}$.

\begin{Lem}
The solution of the inhomogeneous IVP (\ref{raute6}), (\ref{raute6'}) is given by
\begin{align*}
\widehat{Y}(z)&=\widehat{U}(z)\cdot\big(\widehat{Y_0}+\int_{z_0}^ze^{\xi}\,\widehat{U}(\xi)^{-1}\cdot\widehat{h}(\xi)\,d\xi\big)\\
\notag&=\widehat{U}(z)\cdot\widehat{Y_0}+\int_{z_0}^ze^{\xi}\,\widehat{U}(z)\cdot\widehat{U}(\xi)^{-1}\cdot\widehat{h}(\xi)\,d\xi.
\end{align*}
\end{Lem}
\begin{proof}
The map $V(z):=\widehat{U}(z)^{-1}\cdot\widehat{Y}(z)$ satisfies 
\begin{align*}
\frac{d}{dz}V(z)=e^z\,\widehat{U}(z)^{-1}\cdot\widehat{h}(z).
\end{align*}
Hence, the claimed identity follows from the fundamental theorem of calculus.
\end{proof}

\subsection{Discussion of the monodromy map} 
In the present subsection we prove that the point $s=0$ is a removable singularity of the monodromy map $M$.
Furthermore, we examine whether the matrices $M(s)$ and $M(s_0)$ are conjugate for all $s\in\Omega_{\rho}$.

\smallskip

Using (\ref{raute7'}) one proves easily that 
\begin{align*}
M(e^z)=\widehat{U}(z)\cdot M(e^{z_0})\cdot\widehat{U}(z)^{-1}
\end{align*}
and thus the matrices $M(s)$ and $M(s_0)$ are conjugate for $s\neq 0$.
Consequently, for $s\neq 0$ the characteristic polynomials of the monodromy matrices $M(s)$ coincide.
The same holds for the minimal polynomials of the matrices $M(s)$ with $s\neq 0$.

\smallskip

Notice, that $\lim_{x\rightarrow\infty}(e^{z-x})=0$. However,
the difficulty consists in the fact that in general either the limes 
$\lim_{x\rightarrow\infty}\widehat{U}(z-x)$ does not exist or is not an invertible matrix.  
We overcome these difficulties by a more subtle covering construction. In what follows we consider the
differential equation

\begin{align*}
\frac{d}{ds}\widetilde{U}+\frac{1}{s}\,A(s)\cdot\widetilde{U}=0
\end{align*}

along the curves $c_{\sigma}:t\mapsto \sigma\, e^{it}$. We are interested only in the solutions $U(\sigma,t)=\widetilde{U}\circ c_{\sigma}(t)$
with $U(\sigma,0)=\eins$. Hence, we consider the parameter dependent initial value problem

\begin{align}
\label{raute10}
\frac{d}{dt}U(\sigma,t)+i\,A(\sigma\, e^{it})\cdot U(\sigma,t)=0,\hspace{1cm}U(\sigma,0)=\eins.
\end{align}

If $\sigma=e^{z_0}$, then the map $t\mapsto \widehat{U}(z_0+i t)$ solves the initial value problem (\ref{raute10}).
In other words, we get

\begin{align*}
M(\sigma)=U(\sigma,2\pi).
\end{align*}

Since $A$ is holomorphic we get $A(s)=A_0+i\,\widetilde{A_1}(s)$ with $\widetilde{A_1}$ holomorphic.
Plugging in this expression in (\ref{raute10}) we get

\begin{align*}
\frac{d}{dt}U(\sigma,t)+i\,\big(A_0+i\,\widetilde{A_1}(s)\big)\cdot U(\sigma,t)=0,\hspace{1cm} U(\sigma,0)=\eins.
\end{align*}

Now we let converge $\sigma$ against $0$ in the preceding initial value problem.
The solution $U$ depends differentiable on $\sigma$ and it is obvious that

\begin{align*}
U(0,t)=\exp(-i t\,A_0)
\end{align*}

holds. Hence, we have proven the following proposition.

\begin{Prop}
The monodromy map $M$ is continuous and thus holomorphic continuable to $\Omega_{\rho}$. Furthermore, the identities
\begin{align*}
M(0)=U(0,2\pi)=\exp(-2\pi i\,A_0)
\end{align*}
hold.
\end{Prop}

\begin{Cor}
The minimal polynomial of the matrices $M(s)$, $s\neq 0$, annuls $M(0)$.
\end{Cor}

\begin{flushleft}
{\bf{Warning}}: The fact that $M$ is holomorphic continuable to $\Omega_{\rho}$ does not imply that $M(0)$ is conjugate to a matrix $M(\sigma)$ with $\sigma\in\C^{*}$. A well-known counterexample are the nilpotent matrices $M(\sigma)=\left(\begin{array}{cc} 0&\sigma \\ 0&0 \end{array}\right)$,
which are pairwise conjugate for $\sigma\neq 0$.
However, $\lim_{\sigma\rightarrow 0}M(\sigma)=0$, i.e., they converge against a matrix in a different conjugacy class.
\end{flushleft}

\smallskip

\begin{Lem}
The preceding counterexample also works in the context of monodromy.
\end{Lem}
\begin{proof}
For given $\lambda\in\C$ we consider the matrix

\begin{align*}
A(s)=A_0+sA_1,\hspace{0.2cm}\mbox{with}\hspace{0.2cm}A_0=\left(\begin{array}{cc} \lambda&0 \\ 0&\lambda+1 \end{array}\right)\hspace{0.2cm}\mbox{and}\hspace{0.2cm}
A_1=\left(\begin{array}{cc} 0&1 \\ 0&0 \end{array}\right).
\end{align*}

Hence,

\begin{align*}
M(0)=e^{-2\pi i\,\lambda}\cdot\eins
\end{align*}

and the differential equation system for the fundamental solution is given by

\begin{align*}
\frac{d}{dz}\widehat{U}+A_0\cdot\widehat{U}+e^z\,A_1\cdot\widehat{U}=0
\end{align*}

or, equivalently, in terms of the map $\widehat{V}(z)=\exp(zA_0)\cdot\widehat{U}(z)$ by

\begin{align*}
0=\frac{d}{dz}\widehat{V}+\left(\begin{array}{cc} 0&1 \\ 0&0 \end{array}\right)\cdot\widehat{V}.
\end{align*} 

In other words, we get

\begin{align*}
0=\frac{d}{dz}\big( \exp(z\left(\begin{array}{cc} 0&1 \\ 0&0 \end{array}\right))\widehat{V}\big)=\frac{d}{dz}\big(\left(\begin{array}{cc} 1&z \\ 0&1 \end{array}\right)\widehat{V}\big).
\end{align*}

Consequently, the fundamental solution is --- up to a right multiplication with an invertible matrix --- of the form

\begin{align*}
\widehat{U}(z)=\exp(-z\,A_0)\cdot\widehat{V}(z)=\exp(-z\,A_0)\cdot\left(\begin{array}{cc} 1&-z \\ 0&1 \end{array}\right)=
e^{-\lambda\,z}\,\left(\begin{array}{cc} 1&-z \\ 0&e^{-z} \end{array}\right).
\end{align*}

Hence,

\begin{align*}
M(e^z)=\widehat{U}(z+2\pi i)\cdot\widehat{U}(z)^{-1}=e^{-2\pi i\,\lambda}\,\left(\begin{array}{cc} 1&-2\pi i\,e^z \\ 0&1 \end{array}\right),
\end{align*}

and thus

\begin{align*}
M(\sigma)=e^{-2\pi i\,\lambda}\, \left(\begin{array}{cc} 1&-2\pi i\,\sigma \\ 0&1 \end{array}\right).
\end{align*}

Obviously, the matrix $M(0)$ is semi-simple and thus in particular not conjugate to $M(\sigma)$, $\sigma\neq 0$, since these matrices have a nilpotent part.
\end{proof}

\subsection{Geometric properties of the monodromy map}
\label{sub7}
The map

\begin{align*}
\varphi:(z,\widehat{Y})\mapsto (z+2\pi i,M(e^z)\cdot\widehat{Y})
\end{align*}

is fixpoint free and biholomorphic on the domain $\widehat{\pi}^{-1}(\Omega_{\rho}^{\circ})\times\C^n\subset\C^{n+1}$.
Hence, we get a free action of $\mathbb{Z}=\left\langle \varphi\right\rangle$. Consequently, we obtain a $n$-dimensional holomorphic vector bundle

\begin{align*}
E:=(\widehat{\pi}^{-1}(\Omega_{\rho}^{\circ})\times\C^n)/\left\langle \varphi\right\rangle \rightarrow \Omega_{\rho}^{\circ}
\end{align*}

and the fundamental solution $\widehat{U}$ of (\ref{raute6hom}) is a holomorphic trivialization of this vector bundle.

\smallskip

The solutions of the inhomogeneous system (\ref{raute6}) are --- under certain conditions --- in a holomorphic $n$-dimensional
affine bundle $E_{\mbox{aff}}\rightarrow\Omega_{\rho}^{\circ}$, on which the vector bundle $E\rightarrow \Omega_{\rho}^{\circ}$ acts naturally by addition

\begin{center}
    \makebox[0pt]{
\begin{xy}
  \xymatrix{
      A \ar[r]^{+} \ar[d]   &   E_{\mbox{aff}}\ar[d]  \\
      \Omega_{\rho}^{\circ}\ar[r]_{id}   &   \Omega_{\rho}^{\circ}  }
\end{xy}}
\end{center}

where $E\times_{\Omega_{\rho}^{\circ}} E_{\mbox{aff}}=\left\{(\widehat{Y},f)\in E_{\mbox{aff}}\,\lvert\,\pi(\widehat{Y})=\pi(f)\right\}$
denotes the fiber product of these bundles.

\subsection{Non-linear regular-singular systems}
\label{sub8}
Let 
\begin{align}
\label{raute14}
0=\frac{d}{d\xi}Y+\frac{1}{\xi}\,f(\xi,Y),
\end{align}

be a differential system defined on a punctured neighborhood $\Omega_{\rho}=B(0,\rho)-\left\{0\right\}\subset\C$ of $0$
where $f:B(0,\rho)\times\C^n\rightarrow\C^n$ is a differentiable/ holomorphic function.
Such a system is called \textit{a local, non-linear, regular-singular differential equation system}.

\smallskip

Notice, that every such function $f$ possesses a expansion of the form

\begin{align*}
f(\xi,Y)=\sum_{k=0}^{\infty}f_{k}(\xi,Y\,...\,Y),
\end{align*}
 
where we consider $f_k$ as holomorphic map in the $k$-linear maps, i.e., $B(0,\rho)\rightarrow L^k(\C^n\times\, ...\,\times\C^n,\C^n)$.

\smallskip

We call (\ref{raute14}) \textit{weakly non-linear} in the singularity at $z=0$ if and only if

\begin{align*}
f_{k}(\xi,Y\,...\,Y)=\xi\,\widehat{f_{k}}(\xi,Y\,...\,Y)
\end{align*}
for all $k\geq 2$, where $\widehat{f_{k}}:B(0,\rho)\rightarrow L^k(\C^n\times\, ...\,\times\C^n,\C^n)$ is holomorphic.

\smallskip

Here, too we consider the pull back of (\ref{raute14}) under $\exp$

\begin{align*}
\widehat{\Omega_{\rho}}:=\left\{z\in\C\,\lvert\,\mbox{im}(z)<\ln(\rho)\right\}\rightarrow\Omega_{\rho}, z\mapsto\xi=e^{z},
\end{align*}

i.e., we consider the differential equation system

\begin{align*}
0=\frac{d}{dz}\big(Y\circ\exp\big)+f\big(e^z,Y\circ\exp(z)\big).
\end{align*}

Thus we obtain a simply connected domain of definition and, furthermore, we avoid monodromy-problems and multi-valued functions.

\subsubsection{Continuous continuation of the solutions to $\xi=0$}
\begin{Lem}
Let $\xi_0\in\Omega_{\rho}$ and let $Y$ be a solution of (\ref{raute14}) along the line $t\mapsto t\xi_0$, $t\in (0,1]$.
Furthermore, assume that the map $t\mapsto Y(t\,\xi_0)$ is continuous continuable to $t=0$ with $\lim_{t\rightarrow 0}Y(t\,\xi_0)=:Y_0$.
Then,

\begin{align*}
f(0,Y_0)=0
\end{align*}

holds.
\end{Lem}
\begin{proof}
The main theorem of calculus implies

\begin{align*}
Y(t_0\,\xi_0)-Y_0=\int_{0}^{t_0}\frac{1}{t}\,\frac{1}{\xi_0}f(t\,\xi_0,Y(t\,\xi_0))\,dt.
\end{align*}

Suppose that $W_0:=\frac{1}{\xi_0}f(0,Y_0)\neq 0$.
Then, we obtain

\begin{align*}
\lvert \frac{1}{\xi_0}f(t\,\xi_0,Y(t\,\xi_0))-W_0\lvert\,\leq\,\frac{1}{2}\,\lvert W_0\lvert
\end{align*}

for all $t\in\left[0,t_0\right]$ under the condition, that $t_0\in(0,1]$ is sufficiently small.
This in turn would imply

\begin{align*}
\lvert\int_{\theta}^{t_0}\frac{1}{t}\,\frac{1}{\xi_0}f(t\,\xi_0,Y(t\,\xi_0))\,dt\lvert\, &\geq\,\lvert \big(\int_{\theta}^{t_0}\frac{1}{t}\,dt\big)\cdot W_0\lvert-\int_{\theta}^{t_0}\frac{1}{t}\lvert \frac{1}{\xi_0}f(t\,\xi_0,Y(t\,\xi_0))-W_0\lvert\,dt\\
\,&\,\geq \frac{1}{2}\,\ln(\frac{t_0}{\theta})\,\lvert W_0\lvert.
\end{align*}

Consequently, we would get

\begin{align*}
\lim_{\theta\searrow 0}\lvert Y(t_0\,\xi_0)-Y(\theta\,\xi_0)\lvert=\lim_{\theta\searrow 0}\lvert\int_{\theta}^{t_0}\frac{1}{t}\, \frac{1}{\xi_0}f(t\,\xi_0,Y(t\,\xi_0))\,dt\lvert= +\infty.
\end{align*}

However, this contradicts the continuity of the map $\theta\mapsto \lvert Y(t_0\,\xi_0)-Y(\theta\,\xi_0)\lvert$ in $\theta=0$.

\end{proof}

\subsubsection{Normalization of continuous continuable solutions}

Let $Y$ be a solution of (\ref{raute14}) along the line $t\mapsto t\xi_0$, $t\in (0,1]$,
such that the map $t\mapsto Y(t\,\xi_0)$ is continuous continuable to $t=0$ with $\lim_{t\rightarrow 0}Y(t\,\xi_0)=:Y_0$.
Then the map

\begin{align*}
\widetilde{Y}:=Y-Y_0
\end{align*}

also solves a non-linear, regular-singular system with the new right hand side

\begin{align*}
\widetilde{f}(\xi,\widetilde{Y}):=f(\xi,\widetilde{Y}+Y_0)
\end{align*}

and the corresponding expansion terms $\widetilde{f_k}$.

\medskip

The preceding transformation has the following properties:
\begin{enumerate}
	\item $\widetilde{f}(0,0)=f(0,Y_0)=0$,
	\item if $f$ is weakly non-linear, then $\widetilde{f}$ is also weakly non-linear and $$\widetilde{f_1}(0,\widetilde{Y})=f(0,\widetilde{Y})=:A_0\cdot\widetilde{Y}$$ holds.
\end{enumerate}

Consequently, we obtain

\begin{align*}
0=\frac{d}{d\xi}\widetilde{Y}+\frac{1}{\xi}\, A_0\cdot\widetilde{Y}+\widetilde{f_{red}}(\xi,\widetilde{Y}).
\end{align*}

\subsection{Smooth solutions of the non-linear regular-singular system}

\begin{Prop}
\label{satz}
Let $Y$ be a local solution --- i.e., it is defined in a punctured neighborhood of $\xi=0$ --- of the regular-singular system

\begin{align}
\label{stern}
0=\frac{d}{d\xi}Y+\frac{1}{\xi}\,f(\xi,Y),
\end{align}

where $f$ is an analytic function.
Then the following statements hold:

\begin{enumerate}
	\item \label{i} if $Y$ is continuous continuable to $\xi=0$ with $\lim_{\xi\rightarrow 0}Y(\xi)=Y_0$, then 
	\begin{align*}
	f(0,Y_0)=0,
	\end{align*}
	\item \label{ii} if $Y$ is in addition differentiable at the point $\xi=0$ with $\frac{d}{d\xi}Y(0)=Y_1$, or if $Y^{'}$ is continuous continuable to $\xi=0$
	with $\lim_{\xi\rightarrow 0}Y^{'}(\xi)=Y_1$, then we get
	\begin{align*}
	0=a_0+(\eins+A_0)\cdot Y_1
	\end{align*}
	with $a_0:=\frac{\partial f}{\partial \xi}(0,Y_0)$ and $A_0:=\frac{\partial f}{\partial Y}(0,Y_0)$.
\end{enumerate}
\end{Prop}

\begin{Bem}
\label{remark1}
The additional prerequisites in \ref{ii} guarantee each that the function

\begin{align*}
\widehat{Y}(\xi):=\frac{1}{\xi}\, (Y(\xi)-Y_0)=\int_{0}^1Y^{'}(t\,\xi)\,dt,
\end{align*}

which is a priori only defined in a punctured neighborhood of $\xi=0$, is continuous continuable to $\xi=0$.
\end{Bem}

\begin{Lem}
\label{lemma1}
Let $f(0,Y_0)=0$. Then $Y$ is a solution of (\ref{stern}) if and only if $\widehat{Y}$ is a solution of

\begin{align}
\label{sternhut}
0=\frac{d}{d\xi}\widehat{Y}+\frac{1}{\xi}\,\widehat{f}(\xi,\widehat{Y}),
\end{align}

where

\begin{align*}
\widehat{f}(\xi,\widehat{Y}):&=\widehat{Y}+\int_0^1\big[\frac{\partial f}{\partial\xi}(t\xi,Y_0+t\xi\widehat{Y})+\frac{\partial f}{\partial Y}(t\xi,Y_0+t\xi\widehat{Y})\cdot\widehat{Y}\big]\,dt\\
&=a_0+(\eins+A_0)\cdot\widehat{Y}+\xi\,\widehat{f_1}(\xi,\widehat{Y}),
\end{align*}

with $\widehat{f_1}$ analytic.
\end{Lem}

\begin{proof}
Obviously, $\frac{d}{d\xi}Y=\xi\,\frac{d}{d\xi}\widehat{Y}+\widehat{Y}$. Consequently, using $f(0,Y_0)=0$, the identity (\ref{stern})
is equivalent to

\begin{align*}
0=\frac{d}{d\xi}\widehat{Y}+\frac{1}{\xi}\,\big[\widehat{Y}+\int_0^1\frac{1}{\xi}\frac{\partial}{\partial t}(f(t\xi,Y_0+t\xi\widehat{Y}))\,dt\big].
\end{align*}

Then the claim follows immediately by calculating the integrand.
\end{proof}

\begin{Lem}
\label{lemma2}
If $Y$ is bounded and differentiable in a punctured neighborhood $U$ of $\xi=0$ and 
and if
\begin{align*}
0=\frac{d}{d\xi}Y+\frac{1}{\xi}\, h(\xi)
\end{align*}
holds on $U$, where $h$ is a continuous function on a neighborhood of $\xi=0$, then
\begin{align*}
h(0)=0.
\end{align*}
\end{Lem}

\begin{proof}
Assume $h(0)\neq 0$.
The continuity of $h$ in $\xi=0$ implies

\begin{align*}
\lvert h(\xi)-h(0)\lvert\,<\,\frac{1}{2}\,\lvert h(0)\lvert
\end{align*}

for all $\xi\neq 0$ which are sufficiently small. Hence, we get

\begin{align*}
Y(\xi)-Y(t\,\xi)=\int_{t}^1\xi\, Y^{'}(\theta\,\xi)\,d\theta=-\int_{t}^1\frac{1}{\theta}\,h(\theta\xi)\,d\theta
\end{align*}

Consequently, we obtain

\begin{align*}
\lvert Y(\xi)- Y(t\,\xi)\lvert\,&\,\geq\int_t^1\frac{1}{\theta}\,\lvert h(0)\lvert\,d\theta-\int_0^1\frac{1}{\theta}\,\lvert h(\theta\,\xi)-h(0)\lvert\,d\theta\\
\,&\,\geq\lvert h(0)\lvert\,\ln(\frac{1}{t})-\frac{1}{2}\,\lvert h(0)\lvert\,\ln(\frac{1}{t})=\frac{1}{2}\,\lvert h(0)\lvert\, \ln(\frac{1}{t}).
\end{align*}

However, this contradicts the fact that the left hand side is bounded for all $t\in\left(0,1\right)$.

\end{proof}

In what follows we combine the previous results to prove Proposition\,\ref{satz}.

\begin{proof}
The prerequisite in \ref{i} guarantees that $h(\xi):=f(\xi,Y(\xi))$ is continuous in a neighborhood of $\xi=0$.
Therefore \ref{i} is a direct consequence of Lemma\,\ref{lemma2}.\\
Due to Remark\,\ref{remark1} and Lemma\,\ref{lemma1}, the claim \ref{ii} is a consequence of \ref{i}.
\end{proof}

\begin{flushleft}
{\bf{Warning}}:
The continuity of $Y$ at $\xi=0$ in general does not imply that $\widehat{Y}$ has a continuous continuation or equivalently the existence of $Y^{'}$.
An easy counterexample is the differential equation

\begin{align*}
y^{'}=\frac{1}{\xi}\,y-1
\end{align*}

with the at $\xi=0$ continuous solution 

\begin{align*}
y(\xi)=\xi\,(1+\ln(\frac{1}{\xi})).
\end{align*}

However, $y$ not continuous differentiable; the first derivative $y'(\xi)=\ln(\frac{1}{\xi})$ has an essential singularity at $\xi=0$.
Furthermore, this solution has non-trivial monodromy:

\begin{align*}
y(\xi\, e^{2\pi i\,k})=y(\xi)-2\pi i\,k\xi.
\end{align*}

In particular, a solution of the above differential equation which is restricted to the radial segment $t\mapsto t\,\xi$, $0\leq t\leq 1$, is not determined uniquely by the initial value $y(0)=0$.
\end{flushleft}

\begin{question}
Suppose that $Y$ is a continuous solution of (\ref{stern}). Does the existence of the derivative $Y'(0)$ guarantee that $Y^{'}$ is continuous in a neighborhood of $0$?
\end{question}

\begin{Bem}
\begin{enumerate}
	\item If (\ref{stern}) possesses a holomorphic and bounded solution which is defined on a punctured neighborhood 
of $\xi=0$, then this solution is holomorphic in $\xi=0$. This simply follows from Riemann's theorem on removable singularities.
	\item 
	Let a solution $Y$ of (\ref{stern}) with $Y(0)=Y_0$ which is continuous in $0$ be given.
	Suppose there exists a stability condition for the right hand side of (\ref{stern}) which guarantees the existence and uniqueness of 
	such solutions $Y$ along each radial segment $t\mapsto t\,\xi$, $0<t\leq 1$. Such a stability condition is sufficient to prove that the holomorphic
	solution, composed from such radial solutions, does not have monodromy in the punctured neighborhood of $0$.
	\end{enumerate}
\end{Bem}

\begin{flushleft}
{\bf{Remark (Regularity of $f$)}}:
The function $y(t):=e^{-1/t}$, $t\in\C^{*}$, has derivative $y'(t):=1/t^2\,e^{-1/t}$ and hence satisfies the differential equation

\begin{align*}
y^{'}(t)=\frac{1}{t^2}\, y(t).
\end{align*}

Since $\frac{1}{t}=-\ln y(t)$ one can write this differential equation in the following regular-singular form

\begin{align*}
y'(t)=\frac{1}{t}\, f(t,y)
\end{align*}

with $f(t,y)=y\,\ln(\frac{1}{y})=\frac{1}{2}\,y\,\ln(\frac{1}{y^2})$. Obviously, $f$ is continuous at $y=0$. However, $f$ is not continuous differentiable at $y=0$, this is even not the case on the real axis. Hence, in the sense of complex analysis, the point $y=0$ is an essential singularity of $f$.
\end{flushleft}

\smallskip

\begin{flushleft}
{\bf{Consequence}}:
In order to develop a theory of regular-singular systems it is crucial to demand sufficiently many differentiability for $f$ or even that $f$ is analytic.
It is not sufficient to demand continuity of $f$.
\end{flushleft}

\subsection{Solutions with given initial singular value $Y_0$}
\label{sub10}
In what follows we enumerate several \textit{facts}.
\begin{enumerate}
	\item The set of the admissible singular initial values $Y_0$ consists of the solutions of the equation
	\begin{align*}
	f(0,Y_0)=0.
	\end{align*}
	See Proposition\,\ref{satz} \ref{i}.
	\item The holomorphic solutions with given $Y_0$ are in one-to-one correspondence to the continuous/holomorphic solutions $\widehat{Y}$ of (\ref{sternhut}).
	See Proposition\,\ref{satz} \ref{ii}.
	
	\item The structure of the right hand side of (\ref{sternhut}) is much more special than that of the right hand side of (\ref{stern}).
	Indeed, due to Lemma\,\ref{lemma1} the identity (\ref{sternhut}) is given by
	  
	\begin{align*}
	0=\frac{d}{d\xi}\widehat{Y}+\frac{1}{\xi}\,(a_0+(\eins+A_0)\cdot\widehat{Y})+\widehat{f_1}(\xi,\widehat{Y}).
	\end{align*}
	
	Thus the singular term is linear in $\widehat{Y}$.
	
	\item Due to Proposition\,\ref{satz} \ref{ii} only initial values $\widehat{Y_0}=\widehat{Y}(0)=Y'(0)$ with $a_0+(\eins+A_0)\cdot\widehat{Y_0}=0$
	are admissible. Hence, the map $\widetilde{Y}(\xi):=\widehat{Y}(\xi)-\widehat{Y_0}$ is a solution of
	
	\begin{align}
	\label{sterntilde}
	0=\frac{d}{d\xi}\widetilde{Y}+\frac{1}{\xi}\,(\eins+A_0)\cdot\widetilde{Y}+\widetilde{f_1}(\xi,\widetilde{Y})
	\end{align}
	
	with $\widetilde{Y}(0)=0$ and $\widetilde{f}(\xi,\widetilde{Y})=\widehat{f}(\xi,\widehat{Y_0}+\widetilde{Y})$.
	
	\item \label{fact5} The condition $\Re(\lambda)\geq 0$ for all eigenvalues $\lambda$ of $\eins+A_0$ guarantees the existence and uniqueness of 
	the solution $\widetilde{Y}$ of (\ref{sterntilde}) with initial value $\widetilde{Y}(0)$. (??) Consequently, this condition also implies the existence and uniqueness of the solution $\widehat{Y}$ of (\ref{sternhut}) for any admissible initial value $\widehat{Y_0}$.
	
	\item \label{fact6} Each local solution $\widehat{Y}$ of (\ref{sternhut}) with $\widehat{Y}(0)=\widehat{Y_0}$ determined uniquely by the spectral condition considered in the preceding item, translates into a holomorphic solution $Y$ of (\ref{stern}) with $Y(0)=Y_0$ and $Y'(0)=\widehat{Y_0}$.
	
	\item
	The preceding considerations, do not exclude the existence of further solutions $Y$ of (\ref{stern}) with $\lim_{\xi\rightarrow 0}Y(\xi)=Y_0$ which have monodromy. It is true that the latter solutions also correspond to solutions $\widehat{Y}$ of (\ref{sternhut}), but we do not know about their asymptotic behavior. We only know that
	\begin{align*}
	\lim_{\xi\rightarrow 0}\xi\,\widehat{Y}(\xi)=0
	\end{align*}
	is satisfied.
	
\end{enumerate}

\bibliographystyle{plain}
\bibliography{mybib}

@article {MR1214054,
    AUTHOR = {Urakawa, Hajime},
     TITLE = {Equivariant harmonic maps between compact {R}iemannian
              manifolds of cohomogeneity {$1$}},
   JOURNAL = {Michigan Math. J.},
  FJOURNAL = {Michigan Mathematical Journal},
    VOLUME = {40},
      YEAR = {1993},
    NUMBER = {1},
     PAGES = {27--51},
      ISSN = {0026-2285,1945-2365},
   MRCLASS = {58E20 (53C30)},
  MRNUMBER = {1214054},
MRREVIEWER = {Martin\ A.\ Guest},
       DOI = {10.1307/mmj/1029004673},
       URL = {https://doi.org/10.1307/mmj/1029004673},
}

@article {MR1436833,
    AUTHOR = {Bizo\'n, Piotr and Chmaj, Tadeusz},
     TITLE = {Harmonic maps between spheres},
   JOURNAL = {Proc. Roy. Soc. London Ser. A},
  FJOURNAL = {Proceedings of the Royal Society. London. Series A.
              Mathematical, Physical and Engineering Sciences},
    VOLUME = {453},
      YEAR = {1997},
    NUMBER = {1957},
     PAGES = {403--415},
      ISSN = {0962-8444,2053-9169},
   MRCLASS = {58E20},
  MRNUMBER = {1436833},
MRREVIEWER = {Anne-Jo\"elle\ Vanderwinden},
       DOI = {10.1098/rspa.1997.0023},
       URL = {https://doi.org/10.1098/rspa.1997.0023},
}

@incollection {MR2389639,
    AUTHOR = {H\'elein, Fr\'ed\'eric and Wood, John C.},
     TITLE = {Harmonic maps},
 BOOKTITLE = {Handbook of global analysis},
     PAGES = {417--491, 1213},
 PUBLISHER = {Elsevier Sci. B. V., Amsterdam},
      YEAR = {2008},
      ISBN = {978-0-444-52833-9},
   MRCLASS = {58E20 (53C28 53C43 58-02)},
  MRNUMBER = {2389639},
MRREVIEWER = {Andreas\ Gastel},
       DOI = {10.1016/B978-044452833-9.50009-7},
       URL = {https://doi.org/10.1016/B978-044452833-9.50009-7},
}

@misc{BS1,
      title={Stable proper biharmonic maps in Euclidean spheres}, 
      author={Volker Branding and Anna Siffert},
      year={2025},
      eprint={2507.06708},
      archivePrefix={arXiv},
      primaryClass={math.DG},
      url={https://arxiv.org/abs/2507.06708}, 
   JOURNAL = {to appear in Annali della Scuola Normale Superiore di Pisa,
Classe di Scienze},
}

@article {MR4927650,
    AUTHOR = {Branding, Volker and Siffert, Anna},
     TITLE = {Infinite families of harmonic self-maps of ellipsoids in all
              dimensions},
   JOURNAL = {Nonlinear Anal.},
  FJOURNAL = {Nonlinear Analysis. Theory, Methods \& Applications. An
              International Multidisciplinary Journal},
    VOLUME = {261},
      YEAR = {2025},
     PAGES = {Paper No. 113874, 11},
      ISSN = {0362-546X,1873-5215},
   MRCLASS = {58E20 (53C43)},
  MRNUMBER = {4927650},
       DOI = {10.1016/j.na.2025.113874},
       URL = {https://doi.org/10.1016/j.na.2025.113874},
}

@misc{VZ1,
      title={Initial value problems on cohomogeneity one manifolds, I}, 
      author={Luigi Verdiani and Wolfgang Ziller},
      year={2024},
      eprint={2412.06058},
      archivePrefix={arXiv},
      primaryClass={math.DG},
      url={https://arxiv.org/abs/2412.06058}, 
}

@article {MR1298998,
    AUTHOR = {Ding, Wei Yue},
     TITLE = {Harmonic {H}opf constructions between spheres},
   JOURNAL = {Internat. J. Math.},
  FJOURNAL = {International Journal of Mathematics},
    VOLUME = {5},
      YEAR = {1994},
    NUMBER = {6},
     PAGES = {849--860},
      ISSN = {0129-167X,1793-6519},
   MRCLASS = {58E20 (34B15)},
  MRNUMBER = {1298998},
MRREVIEWER = {John\ C.\ Wood},
       DOI = {10.1142/S0129167X94000437},
       URL = {https://doi.org/10.1142/S0129167X94000437},
}

@article {MR3427685,
    AUTHOR = {Siffert, Anna},
     TITLE = {Infinite families of harmonic self-maps of spheres},
   JOURNAL = {J. Differential Equations},
  FJOURNAL = {Journal of Differential Equations},
    VOLUME = {260},
      YEAR = {2016},
    NUMBER = {3},
     PAGES = {2898--2925},
      ISSN = {0022-0396,1090-2732},
   MRCLASS = {58E20 (34B15 55M25)},
  MRNUMBER = {3427685},
MRREVIEWER = {Giandomenico\ Orlandi},
       DOI = {10.1016/j.jde.2015.10.023},
       URL = {https://doi.org/10.1016/j.jde.2015.10.023},
}

@article {MR3745872,
    AUTHOR = {Siffert, Anna},
     TITLE = {Harmonic self-maps of {${\rm SU}(3)$}},
   JOURNAL = {J. Geom. Anal.},
  FJOURNAL = {Journal of Geometric Analysis},
    VOLUME = {28},
      YEAR = {2018},
    NUMBER = {1},
     PAGES = {587--605},
      ISSN = {1050-6926,1559-002X},
   MRCLASS = {58E20 (34B15 55M25)},
  MRNUMBER = {3745872},
MRREVIEWER = {Song-il\ Ri},
       DOI = {10.1007/s12220-017-9833-0},
       URL = {https://doi.org/10.1007/s12220-017-9833-0},
}

@article {MR2022387,
    AUTHOR = {Gastel, Andreas},
     TITLE = {On the harmonic {H}opf construction},
   JOURNAL = {Proc. Amer. Math. Soc.},
  FJOURNAL = {Proceedings of the American Mathematical Society},
    VOLUME = {132},
      YEAR = {2004},
    NUMBER = {2},
     PAGES = {607--615},
      ISSN = {0002-9939,1088-6826},
   MRCLASS = {58E20 (34B15)},
  MRNUMBER = {2022387},
MRREVIEWER = {John\ C.\ Wood},
       DOI = {10.1090/S0002-9939-03-07062-X},
       URL = {https://doi.org/10.1090/S0002-9939-03-07062-X},
}

@article {MR2480860,
    AUTHOR = {P\"{u}ttmann, Thomas},
     TITLE = {Cohomogeneity one manifolds and self-maps of nontrivial
              degree},
   JOURNAL = {Transform. Groups},
  FJOURNAL = {Transformation Groups},
    VOLUME = {14},
      YEAR = {2009},
    NUMBER = {1},
     PAGES = {225--247},
      ISSN = {1083-4362},
   MRCLASS = {57S15 (55M25)},
  MRNUMBER = {2480860},
MRREVIEWER = {Sergio Console},
       DOI = {10.1007/s00031-008-9037-6},
       URL = {https://doi.org/10.1007/s00031-008-9037-6},
}

@article {MR1923478,
    AUTHOR = {Grove, Karsten and Ziller, Wolfgang},
     TITLE = {Cohomogeneity one manifolds with positive {R}icci curvature},
   JOURNAL = {Invent. Math.},
  FJOURNAL = {Inventiones Mathematicae},
    VOLUME = {149},
      YEAR = {2002},
    NUMBER = {3},
     PAGES = {619--646},
      ISSN = {0020-9910,1432-1297},
   MRCLASS = {53C21},
  MRNUMBER = {1923478},
MRREVIEWER = {Burkhard\ Wilking},
       DOI = {10.1007/s002220200225},
       URL = {https://doi.org/10.1007/s002220200225},
}

@article {MR4400726,
    AUTHOR = {Verdiani, L. and Ziller, W.},
     TITLE = {Smoothness conditions in cohomogeneity one manifolds},
   JOURNAL = {Transform. Groups},
  FJOURNAL = {Transformation Groups},
    VOLUME = {27},
      YEAR = {2022},
    NUMBER = {1},
     PAGES = {311--342},
      ISSN = {1083-4362,1531-586X},
   MRCLASS = {53C30},
  MRNUMBER = {4400726},
MRREVIEWER = {Reza\ Mirzaie},
       DOI = {10.1007/s00031-020-09618-9},
       URL = {https://doi.org/10.1007/s00031-020-09618-9},
}

@incollection {MR1155662,
    AUTHOR = {Alekseevski\u i, A. V. and Alekseevski\u i, D. V.},
     TITLE = {{$G$}-manifolds with one-dimensional orbit space},
 BOOKTITLE = {Lie groups, their discrete subgroups, and invariant theory},
    SERIES = {Adv. Soviet Math.},
    VOLUME = {8},
     PAGES = {1--31},
 PUBLISHER = {Amer. Math. Soc., Providence, RI},
      YEAR = {1992},
      ISBN = {0-8218-4107-6},
   MRCLASS = {57S15},
  MRNUMBER = {1155662},
MRREVIEWER = {Yoshinobu\ Kamishima},
       DOI = {10.1007/bf01084048},
       URL = {https://doi.org/10.1007/bf01084048},
}

@book {MR3362465,
    AUTHOR = {Alexandrino, Marcos M. and Bettiol, Renato G.},
     TITLE = {Lie groups and geometric aspects of isometric actions},
 PUBLISHER = {Springer, Cham},
      YEAR = {2015},
     PAGES = {x+213},
      ISBN = {978-3-319-16612-4; 978-3-319-16613-1},
   MRCLASS = {22-01 (22E46 22F05 53-01)},
  MRNUMBER = {3362465},
MRREVIEWER = {Vladimir\ V.\ Gorbatsevich},
       DOI = {10.1007/978-3-319-16613-1},
       URL = {https://doi.org/10.1007/978-3-319-16613-1},
}

@article {MR1758585,
    AUTHOR = {Eschenburg, J.-H. and Wang, McKenzie Y.},
     TITLE = {The initial value problem for cohomogeneity one {E}instein
              metrics},
   JOURNAL = {J. Geom. Anal.},
  FJOURNAL = {The Journal of Geometric Analysis},
    VOLUME = {10},
      YEAR = {2000},
    NUMBER = {1},
     PAGES = {109--137},
      ISSN = {1050-6926,1559-002X},
   MRCLASS = {53C25 (34A12 34C60 34E05 53C30)},
  MRNUMBER = {1758585},
MRREVIEWER = {Megan\ M.\ Kerr},
       DOI = {10.1007/BF02921808},
       URL = {https://doi.org/10.1007/BF02921808},
}

@article {MR800005,
    AUTHOR = {Ferus, Dirk and Karcher, Hermann},
     TITLE = {Nonrotational minimal spheres and minimizing cones},
   JOURNAL = {Comment. Math. Helv.},
  FJOURNAL = {Commentarii Mathematici Helvetici},
    VOLUME = {60},
      YEAR = {1985},
    NUMBER = {2},
     PAGES = {247--269},
      ISSN = {0010-2571,1420-8946},
   MRCLASS = {53C42 (53A10)},
  MRNUMBER = {800005},
MRREVIEWER = {Shigeo\ Akiba},
       DOI = {10.1007/BF02567412},
       URL = {https://doi.org/10.1007/BF02567412},
}

@article {MR85460,
    AUTHOR = {Mostert, Paul S.},
     TITLE = {On a compact {L}ie group acting on a manifold},
   JOURNAL = {Ann. of Math. (2)},
  FJOURNAL = {Annals of Mathematics. Second Series},
    VOLUME = {65},
      YEAR = {1957},
     PAGES = {447--455},
      ISSN = {0003-486X},
   MRCLASS = {17.0X},
  MRNUMBER = {85460},
MRREVIEWER = {A.\ Shields},
       DOI = {10.2307/1970056},
       URL = {https://doi.org/10.2307/1970056},
}

@article {MR4269412,
    AUTHOR = {Foscolo, Lorenzo and Haskins, Mark and Nordstr\"om, Johannes},
     TITLE = {Infinitely many new families of complete cohomogeneity one
              {$\rm G_2$}-manifolds: {$\rm G_2$} analogues of the
              {T}aub-{NUT} and {E}guchi-{H}anson spaces},
   JOURNAL = {J. Eur. Math. Soc. (JEMS)},
  FJOURNAL = {Journal of the European Mathematical Society (JEMS)},
    VOLUME = {23},
      YEAR = {2021},
    NUMBER = {7},
     PAGES = {2153--2220},
      ISSN = {1435-9855,1435-9863},
   MRCLASS = {53C29},
  MRNUMBER = {4269412},
MRREVIEWER = {Lorenz\ J.\ Schwachh\"ofer},
       DOI = {10.4171/jems/1051},
       URL = {https://doi.org/10.4171/jems/1051},
}

@article {MR95897,
    AUTHOR = {Mostert, Paul S.},
     TITLE = {Errata, ``{O}n a compact {L}ie group acting on a manifold''},
   JOURNAL = {Ann. of Math. (2)},
  FJOURNAL = {Annals of Mathematics. Second Series},
    VOLUME = {66},
      YEAR = {1957},
     PAGES = {589},
      ISSN = {0003-486X},
   MRCLASS = {22.00},
  MRNUMBER = {95897},
MRREVIEWER = {A.\ Shields},
       DOI = {10.2307/1969911},
       URL = {https://doi.org/10.2307/1969911},
}

@book {MR2682403,
    AUTHOR = {Kristensson, Gerhard},
     TITLE = {Second order differential equations},
      NOTE = {Special functions and their classification},
 PUBLISHER = {Springer, New York},
      YEAR = {2010},
     PAGES = {xii+219},
      ISBN = {978-1-4419-7019-0},
   MRCLASS = {34-01 (33-01 34M35)},
  MRNUMBER = {2682403},
MRREVIEWER = {Wolfgang\ B\"uhring},
       DOI = {10.1007/978-1-4419-7020-6},
       URL = {https://doi.org/10.1007/978-1-4419-7020-6},
}

@article {MR2975151,
    AUTHOR = {Mitschi, Claude and Singer, Michael F.},
     TITLE = {Monodromy groups of parameterized linear differential
              equations with regular singularities},
   JOURNAL = {Bull. Lond. Math. Soc.},
  FJOURNAL = {Bulletin of the London Mathematical Society},
    VOLUME = {44},
      YEAR = {2012},
    NUMBER = {5},
     PAGES = {913--930},
      ISSN = {0024-6093,1469-2120},
   MRCLASS = {34M56 (12H05)},
  MRNUMBER = {2975151},
MRREVIEWER = {V.\ V.\ Chueshev},
       DOI = {10.1112/blms/bds021},
       URL = {https://doi.org/10.1112/blms/bds021},
}

@misc{mon,
      title={Differential Equations and Monodromy}, 
      author={Tyakal N. Venkataramana},
      year={2019},
      eprint={1911.02840},
      archivePrefix={arXiv},
      primaryClass={math.CA},
      url={https://arxiv.org/abs/1911.02840}, 
}

@book {MR2961944,
    AUTHOR = {Teschl, Gerald},
     TITLE = {Ordinary differential equations and dynamical systems},
    SERIES = {Graduate Studies in Mathematics},
    VOLUME = {140},
 PUBLISHER = {American Mathematical Society, Providence, RI},
      YEAR = {2012},
     PAGES = {xii+356},
      ISBN = {978-0-8218-8328-0},
   MRCLASS = {34-01 (37-01 39-01)},
  MRNUMBER = {2961944},
MRREVIEWER = {Eleonora\ Catsigeras},
       DOI = {10.1090/gsm/140},
       URL = {https://doi.org/10.1090/gsm/140},
}

@article {MR4000241,
    AUTHOR = {P\"{u}ttmann, Thomas and Siffert, Anna},
     TITLE = {Harmonic self-maps of cohomogeneity one manifolds},
   JOURNAL = {Math. Ann.},
  FJOURNAL = {Mathematische Annalen},
    VOLUME = {375},
      YEAR = {2019},
    NUMBER = {1-2},
     PAGES = {247--282},
      ISSN = {0025-5831},
   MRCLASS = {58E20 (34B15 55M25 57S15)},
  MRNUMBER = {4000241},
MRREVIEWER = {Andreas Gastel},
       DOI = {10.1007/s00208-019-01848-x},
       URL = {https://doi.org/10.1007/s00208-019-01848-x},
}

@article {MR3357596,
    AUTHOR = {Montaldo, S. and Oniciuc, C. and Ratto, A.},
     TITLE = {Rotationally symmetric biharmonic maps between models},
   JOURNAL = {J. Math. Anal. Appl.},
  FJOURNAL = {Journal of Mathematical Analysis and Applications},
    VOLUME = {431},
      YEAR = {2015},
    NUMBER = {1},
     PAGES = {494--508},
      ISSN = {0022-247X},
   MRCLASS = {58E20 (35B35 35J50 35R01)},
  MRNUMBER = {3357596},
MRREVIEWER = {Andreas Gastel},
       DOI = {10.1016/j.jmaa.2015.05.082},
       URL = {https://doi.org/10.1016/j.jmaa.2015.05.082},
}

@book {MR2044031,
    AUTHOR = {Baird, Paul and Wood, John C.},
     TITLE = {Harmonic morphisms between {R}iemannian manifolds},
    SERIES = {London Mathematical Society Monographs. New Series},
    VOLUME = {29},
 PUBLISHER = {The Clarendon Press, Oxford University Press, Oxford},
      YEAR = {2003},
     PAGES = {xvi+520},
      ISBN = {0-19-850362-8},
   MRCLASS = {53C43 (53C21 58E20)},
  MRNUMBER = {2044031},
MRREVIEWER = {Eric Loubeau},
       DOI = {10.1093/acprof:oso/9780198503620.001.0001},
       URL = {https://doi.org/10.1093/acprof:oso/9780198503620.001.0001},
}

\end{document}